\documentclass[letterpaper,11pt,oneside]{amsart}

\usepackage{amsmath}
\usepackage{amssymb}
\usepackage{graphicx}
\usepackage{subcaption}
\usepackage{amsthm}
\usepackage{bbm}

\usepackage{hyperref}
\usepackage{color}
\usepackage{multirow,verbatim}
\usepackage{mathrsfs}

\usepackage{algorithm}
\usepackage{algorithmic}

\usepackage{soul}
\setstcolor{magenta}

\theoremstyle{definition}\newtheorem{proposition}{Proposition}
\theoremstyle{definition}\newtheorem{definition}{Definition}
\theoremstyle{definition}\newtheorem{theorem}{Theorem}
\theoremstyle{definition}\newtheorem{lemma}{Lemma}
\theoremstyle{definition}\newtheorem{remark}{Remark}
\numberwithin{equation}{section}
\numberwithin{theorem}{section}
\numberwithin{definition}{section}
\numberwithin{lemma}{section}
\numberwithin{figure}{section}

\newcommand{\rd}{\mathrm{d}}

\newcommand{\CalE}{\mathcal{E}}
\newcommand{\TOL}{\text{TOL}}
\newcommand{\Cov}{\text{Cov}}
\newcommand{\veps}{\varepsilon}
\DeclareMathOperator{\Tr}{Tr}
\DeclareMathOperator{\Span}{Span}
\DeclareMathOperator{\divop}{div}

\newcommand{\Amat}{\mathsf{A}}
\newcommand{\Bmat}{\mathsf{B}}
\newcommand{\Cmat}{\mathsf{C}}
\newcommand{\Rmat}{\mathsf{R}}
\newcommand{\Xmat}{\mathsf{X}}
\newcommand{\Ymat}{\mathsf{Y}}
\newcommand{\Zmat}{\mathsf{Z}}
\newcommand{\Smat}{\mathsf{S}}

\newcommand{\Qmat}{\mathsf{Q}}

\newcommand{\Umat}{\mathsf{U}}
\newcommand{\Imat}{\mathsf{I}}
\newcommand{\Pmat}{\mathsf{P}}

\newcommand{\Xcal}{\mathcal{X}}
\newcommand{\Ycal}{\mathcal{Y}}
\newcommand{\Zcal}{\mathcal{Z}}

\newcommand{\qlresolved}[1]{}

\newcommand{\jlresolved}[1]{}

\newcommand{\sjwresolved}[1]{}

\newcommand{\jledit}[1]{\textcolor{black}{#1}}
\newcommand{\steveedit}[1]{\textcolor{black}{#1}}
\newcommand{\keedit}[1]{\textcolor{black}{#1}}

\newcommand{\blueedit}[1]{\textcolor{black}{#1}}

\graphicspath{{../}}

\title[Randomized sampling for gFEM]{Randomized sampling for basis functions construction in
  generalized finite element methods}

\author{Ke Chen} 
\address{Mathematics Department, University of Wisconsin-Madison, 480 Lincoln Dr., Madison, WI 53706 USA.}
\email{ke@math.wisc.edu}
\author{Qin Li} 
\address{Mathematics Department, University of Wisconsin-Madison, 480 Lincoln Dr., Madison, WI 53706 USA.}
\email{qinli@math.wisc.edu}
\author{Jianfeng Lu}
\address{Department of Mathematics, Department of Physics and Department of Chemistry, Duke University, Box 90320, Durham, NC 27708 USA.}
\email{jianfeng@math.duke.edu}
\author{Stephen J. Wright} 
\address{Computer Sciences Department, University of Wisconsin-Madison, 1210 W Dayton St, Madison, WI 53706 USA.}
\email{swright@cs.wisc.edu}
\date{\today}

\thanks{The work of Q.L.~is supported in part by National Science Foundation under the grant
  DMS-1619778 and DMS-1750488. The work of J.L.~is supported in part by the National
  Science Foundation under grant DMS-1454939. The work of S.W.~is
  supported in part by NSF Awards IIS-1447449, 1628384, and 1634597
  and AFOSR Award FA9550-13-1-0138. K.C.,  Q.L., and S.W. are supported by TRIPODS: NSF 1740707.}

\begin{document}

\begin{abstract}
  In the \jledit{framework} of generalized finite element methods for
  elliptic equations with rough coefficients, efficiency and accuracy
  of the numerical method depend critically on the use of appropriate
  basis functions. This work explores several random sampling
  strategies \steveedit{that construct approximations to the} optimal
  \steveedit{set of basis functions of a given dimension}, and
  proposes a quantitative criterion to analyze and compare these
  sampling strategies.  \blueedit{Numerical evidence shows that the
    best results are achieved by two strategies, Random Gaussian and
    Smooth boundary sampling.}
\end{abstract}

\maketitle

\section{Introduction} \label{sec:intro}


\keedit{This paper considers techniques for constructing basis
  functions for generalized finite element methods applied to elliptic
  equations with rough coefficients. The elliptic partial differential
  equation is}
\begin{equation} \label{eq:elliptical}
\begin{cases}
-\divop\left(a(x)\nabla u(x)\right) = f(x)\,,\quad  & x\in\Omega;\\
u(x) = 0\,,\quad & x\in\partial\Omega,
\end{cases}
\end{equation}
with $f \in L^2(\Omega)$ and a \jledit{uniformly elliptic} coefficient
function $a \in L^{\infty}(\Omega)$, that is, there exist
$\alpha_{\ast}, \beta_{\ast} > 0$ such that
$a(x) \in [\alpha_{\ast}, \beta_{\ast}]$ for all $x \in \Omega$. Note
that we assume only $L^{\infty}$ regularity of $a$, so the coefficient
could be rather rough, which poses challenges to conventional
numerical methods, such as the standard finite element method with
local polynomial basis functions.

Numerical methods can be designed to take advantage of certain
analytical properties of the problem \eqref{eq:elliptical}. A
classical example is when $a$ is two-scale, that is,
$a(x) = a_0(x, \tfrac{x}{\veps})$ where $a_0(x, y)$ is $1$-periodic
with respect to its second argument. (Thus, $\veps$ characterizes
explicitly the small scale of the problem.) Using the theory of
homogenization~\cite{BLP,Papa75}, several numerical methods have been
proposed over the past decades to capture the homogenized solution of
the problem and possibly also to provide some microscopic
information. Approaches of this type include the multiscale finite
element method~\cite{EHW:00,HW_numerics:99,HWC99,HW97} and the
heterogeneous multiscale
method~\cite{EEngquist:03,EMingZhang:2005,MY06}.

While \steveedit{methods designed for numerical homogenization can be
  applied} to the cases of rough media ($a \in L^{\infty}$), the lack
of favorable structural properties often degrades the efficiency and
convergence rates. Various numerical methods have been proposed for
$L^{\infty}$ media, including the generalized finite element method
\cite{babuska_partition_1997}, upscaling based on harmonic coordinate
\cite{OZ07}, elliptic solvers based on $\mathcal{H}$-matrices
\cite{Bebendorf2007, Hackbusch2015}, and Bayesian numerical
homogenization \cite{Owhadi_bayesian}, to name just a few. Our work
\steveedit{uses the framework of the generalized finite element method
  (gFEM) of~\cite{babuska_partition_1997}}. The idea is to approximate
the local solution space by constructing good basis functions locally
and to use either the partition-of-unity or the discontinuous Galerkin
method to obtain a global discretization.

According to the partition-of-unity finite-element theory, which we
will recall briefly in Section~\ref{sec:preliminary}, the global error
is controlled by the accuracy of the numerical local solution
spaces. Thus, global performance of the method depends critically on
efficient preparation of accurate local solution spaces.
Towards this end, Babu\v{s}ka and Lipton \cite{BL11} studied the
Kolmogorov width of a finite dimensional approximation to the optimal
solution space, and showed that the Kolmogorov width decays almost
exponentially fast, as we \jledit{will} recall in
Section~\ref{sec:preliminary}.
The basis construction algorithm proposed in \cite{BL11} follows the
analysis closely: A full list of $a$-harmonic functions (up to
discretization) is obtained in each patch, and local basis functions
are obtained by a ``post-processing'' step of solving a generalized
eigenvalue problem to select modes with highest ``energy ratios''.
Since the roughness of $a$ necessitates a fine discretization in each
patch, and thus a large number of $a$-harmonic functions per patch,
the overall computational cost of this strategy \jledit{to construct
  local basis functions} is high.

Our work is based on the gFEM framework~\cite{babuska_partition_1997}
together with the concept of optimal local solution space via
Kolmogorov width studied in~\cite{BL11}. 
\blueedit{The idea of introducing random sampling or oversampling to construct local basis functions were studied in \cite{calo2016randomized,GE10,LLYE:12}, and they are shown to be computationally effective. However, a systematic investigation of random sampling in the contexts of numerical PDEs is in lack. There is no criterion that justifies the ``goodness" of basis functions constructed through random sampling. The main contribution of this paper is two-folded. We systematically examine these random sampling approaches and introduce a criterion that evaluates different sets of basis functions, and we furthermore propose a random projection method that obtains a set of $a$-harmonic basis functions automatically.}
Randomized algorithms have been shown to be
powerful in reducing computational complexities in looking for low
rank \jledit{factorization} of matrices. Since the generalized
eigenvalues decay almost exponentially, the local solution space is of
\steveedit{approximate low rank, and random sampling approaches can
  capture this space effectively.}
The efficiency of the approach certainly depends on the particular
random sampling strategy employed; we explore several strategies and
identify the most successful ones.

\blueedit{As mentioned above, the idea of random sampling or oversampling to
  construct local basis functions is not completely
  new. In~\cite{GE10}, the authors proposed to compute a generalized
  eigenvalue problem (using the stiffness matrix and the mass matrix)
  as a post-processing step for basis selection. Similar strategies
  have been considered in the discontinuous Galerkin
  framework~\cite{LLYE:12}, but these approaches require a full basis
  of local solutions.  The random sampling strategy is incorporated
  in~\cite{calo2016randomized} to improve efficiency in the offline
  stage.}

%

\blueedit{There are several important differences between our
  approaches and those of \cite{GE10,LLYE:12,calo2016randomized}.
  First, we provide a quantitative criterion for evaluating the
  efficiency and accuracy of different random sampling
  strategies. Second, we find that the best randomized sampling
  strategy is {\em not} necessarily based on randomly assigning
  boundary conditions, as done in~\cite{calo2016randomized}. As
  indicated by the proposed criterion, a good sampling strategy should
  eliminate boundary layers and maintain much of the ``energies'' of
  the samples in the interior. Third, instead of using stiffness-mass
  ratio as done in~\cite{GE10}, the selection process here is guided
  by the behavior of the restriction operator (see
  Section~\ref{sec:preliminary}), \blueedit{which is proved to be
    optimal in~\cite{BL11}.}  }

Other basis construction approaches based on gFEM framework have been
explored in the literature, mostly based on a similar offline-online
strategy. 
In the offline step, one prepares the solution space (either
local or global). In the online step, one assembles the basis through
the Galerkin framework (see, for example, \cite{MP14, OZ14,
  HouZhang17}). The random sampling strategy can be also explored in
these contexts. 


The organization of the rest of the paper is as follows. We review
preliminaries in Section~\ref{sec:preliminary}, including the basics
of basis construction and error analysis. In
Section~\ref{sec:methods}, we describe the random sampling framework,
and present a few particular sampling strategies.
We connect and compare our framework with the randomized singular
value decomposition (rSVD) in Section~\ref{sec:svd}. To compare the
various sampling approaches, we propose a criterion in
Section~\ref{sec:criteria}, 
\keedit{
according to which random sampling strategy with higher energies achieve smaller Kolmogorov distances to the optimal basis.}
\steveedit{Numerical examples in
  Section~\ref{sec:numerics} demonstrate the effectiveness of our
  approach.}

\blueedit{This paper only serves as the first step towards evaluating randomly constructed basis and there are many other choices and parameters that we do not fully investigate. One example is the ratio of the enlargement: a bigger enlarged domain gives faster decay in singular values, but the numerical cost is fairly high. These are left to future works.}

\section{\steveedit{Previous Results and Context}}\label{sec:preliminary}

Here we provide some preliminary results about the generalized finite
element method, \steveedit{including the concept of low-rank solution
  space, and review the construction of basis functions for the local
  solution space.}

\keedit{We restate the elliptic equation \eqref{eq:elliptical} as
  follows:}
\begin{equation}\label{eqn:elliptic}
\begin{cases}
\mathcal{L}u = -\text{div}\left(a(x)\nabla u(x)\right) = f(x)\,,\quad x\in\Omega;\\
u(x) = 0\,,\quad x\in\partial\Omega,
\end{cases}
\end{equation}
with $0<\alpha_\ast\leq a(x)\leq \beta_\ast$, where $\mathcal{L}$
denotes the elliptic operator. The weak formulation of
\eqref{eqn:elliptic} is
\begin{equation*}
  \langle a(x)\nabla u\,,\nabla v\rangle_{L^2(\Omega)} = \langle
  f\,,v\rangle_{L^2(\Omega)},
\end{equation*}
for all test functions $v$, where $\langle
f,g\rangle_{L^2(\Omega)} := \int_\Omega f(x)g(x)\rd{x}$.

In the Galerkin framework, one constructs the solution space
first. Given the following approximation space, \keedit{defined by
  basis functions $\phi_i$, $i=1,2,\dotsc,n$:
\begin{equation}\label{eqn:space}
\Span\{\phi_i\,,i=1,2,\dotsc n\}\,,
\end{equation}
we substitute the ansatz $U = \sum_{i=1}^nc_i\phi_i\,$
into~\eqref{eqn:elliptic} to obtain:}
\[
\sum_{j=1}^n\langle a(x)\nabla \phi_j\,,\nabla\phi_i\rangle_{L^2(\Omega)} c_j = \langle f\,,\phi_i\rangle_{L^2(\Omega)}\,.
\]
We write this system in matrix form as follows:
\begin{equation}\label{eqn:Acb}
\Amat\vec{c}=\vec{b}\,,
\end{equation}
where $\Amat$ is a symmetric matrix with entries $\Amat_{mn} = \langle
a\nabla\phi_m\,,\nabla\phi_n\rangle_{L^2(\Omega)}$, and $\vec{c}$
(with $\vec{c}_m = c_m$) is a list of coefficients to be determined. 
The right hand side is the load vector $\vec{b}$, with entries
$\vec{b}_m = \langle f\,,\phi_m\rangle_{L^2(\Omega)}$.

It is well known that the following quasi-optimality holds:
\[
\|u - U\|_{\CalE(\Omega)} \leq C\|u-\mathcal{P}u\|_{\CalE(\Omega)}\,,
\]
where $C$ is some constant depending on $\alpha_\ast$ and
$\beta_\ast$, and $\mathcal{P}u$ is the projection of the true
solution $u$ onto the space~\eqref{eqn:space}. Here the energy norm on
any subdomain $\omega \subset \Omega$ is defined by
\begin{equation} \label{eq:Enorm} \|v\|_{\CalE(\omega)} =
  \langle
  a\nabla v\,,\nabla v\rangle_{L^2(\omega)}^{1/2} := \left[
    \int_{\omega} a|\nabla v(x)|^2 \rd{x} \right]^{1/2}.
\end{equation}
Thus to guarantee small numerical error $\|u-U\|_{\CalE(\Omega)}$, we
require a set of basis functions that form a space for which
$\|u-\mathcal{P}u\|_{\CalE(\Omega)}$ is small.

The main difficulty of computing the elliptic equation with rough
coefficient is that a large number of basis functions is apparently
needed. When $a(x)$ is rough with $\veps$ as its smallest scale, for
standard piecewise affine finite elements, the mesh size $\Delta x$
needs
to resolve the smallest
scale, so that $\Delta x\ll\veps$ in each dimension. It follows that the
dimension $n$ of the system \eqref{eqn:Acb} is
$n =\mathcal{O}(1/\veps^d)\gg 1$, where $d$ is the spatial
dimension. The large size of stiffness matrix $\Amat$ and its large
condition number
(usually on the order of $\mathcal{O}(1/\varepsilon^{2})$)
make the problem expensive to solve using this approach.

The question is then whether it is possible to design a Galerkin space
for which $n$ is independent of $\veps$?  As mentioned in
Section~\ref{sec:intro}, the offline-online procedure makes this
approach feasible, as we discuss next.

\subsection{Generalized Finite Element Method}

The generalized Finite Element Method was one of the earliest methods
to utilize the offline-online procedure.  This approach is based on
the partition of unity. One first decomposes the domain $\Omega$ into
many small patches $\omega_i$, $i=1,2,\dotsc,m$,
that form an open cover of $\Omega$. Each patch $\omega_i$ is assigned
a partition-of-unity function $\nu_i$ that is zero outside $\omega_i$
and $1$ over most of the set $\omega_i$. Specifically, there are
positive constant $C$ such that
\begin{subequations} \label{eqn:PoU}
\begin{alignat}{2}
0 \le \nu_i(x) & \le 1, \quad && \mbox{for all $x \in \Omega$ and all
  $i=1,2,\dotsc,m$}, \\ 
\nu_i(x) &= 0, && \mbox{for all $x \in \Omega \setminus \omega_i$ and all $i=1,2,\dotsc,m$}, \\
\label{eq:PoU.3}
\max_{x \in \Omega} \, | \nabla \nu_i(x) | & \le \frac{C}{\text{diam}(\omega_i)}, \quad &&
\mbox{for all $i=1,2,\dotsc,m$}.
\end{alignat}
\end{subequations}
Moreover, we have
\begin{equation} \label{eq:PoU.sum}
\sum_{i=1}^m\nu_i(x) = 1, \quad  \mbox{for all $x \in \Omega$\,.}
\end{equation}

In the offline step, basis functions \keedit{$\phi_{i,j}$,
  $i=1,2,\dotsc,m$, $j=1,2,\dotsc,n_i$ are constructed for each patch
  $\omega_i$, where $n_i$ is the number of basis functions in patch
  $i$.} We denote the numerical local solution space in patch
$\omega_i$ by:
\begin{equation}\label{eqn:local_space}
\Phi_{[i]} = \Span\{\phi_{i,j} \,,j = 1,2,\dotsc,n_i\}\,.
\end{equation}
In the online step, the Galerkin formulation is used, with the space
in~\eqref{eqn:space} replaced by:
\begin{equation}\label{eqn:global_space}
\Phi := \bigoplus_{i=1,2,\dotsc,m} \Phi_{[i]} \nu_i= \Span\{\phi_{i,j}\nu_i \,,i = 1,2,\dotsc,m, \; j = 1,2,\dotsc,n_i\}\,.
\end{equation}
Details can be found in~\cite{babuska_partition_1997}.

The total number of basis functions is $\sum_{i=1}^mn_i$. If all
$n_i$, $i=1,2,\dotsc,m$ are bounded by a modest constant, the
dimension of the space $\Phi$ is of order $m$, so the computation in
the online step is potentially inexpensive. It is proved
in~\cite{babuska_partition_1997} that the total approximation error is
governed by the sum of all local approximation errors.

\begin{theorem}\label{thm:global_error}
Denote by $u$ the solution to~\eqref{eqn:elliptic}. Suppose $\{
\omega_i \}_{i=1,2,\dotsc,m}$ forms an open cover of $\Omega$ and let
$\{ \nu_i \}_{i=1,2,\dotsc,m}$ denote the set of partition-of-unity
functions 
 defined in~\eqref{eqn:PoU}. \steveedit{If the solution can be
   approximated well by $\zeta_i\in \Phi_{[i]}$ in each patch
   $\omega_i$,
the global error is small too.} Specifically, if we assume that
\begin{equation}
\| u-\zeta_i\|_{L^2(\omega_i)} \leq \varepsilon_1(i) \quad \text{and}
\quad\|u-\zeta_i\|_{\mathcal{E}(\omega_i)}\leq \varepsilon_2(i)\,, \quad
i=1,2,\dotsc,m,
\end{equation}
and define
\[
\zeta(x) = \sum_{i=1}^m \zeta_i(x)\nu_i(x)\,,
\]
then $\zeta(x)\in H^1(\Omega)$, and for the constant $C$
defined in~\eqref{eqn:PoU}, we have
\begin{equation*}
\|u - \zeta\|_{L^2(\Omega)}\leq \max_i \|\nu_i\|_\infty \left(\sum_{i=1}^m \varepsilon_1^2(i) \right)^{1/2}\,
\end{equation*}
and
\begin{equation*}
\|u - \zeta\|_{\mathcal{E}(\Omega)}\leq C \left(\sum_{i=1}^m \frac{\varepsilon_1^2(i)}{\text{diam}^2(\omega_i)} +\max_i \|\nu_i\|_\infty^2\sum_{i=1}^m\varepsilon_2^2(i)\right)^{1/2}\,.
\end{equation*}
\end{theorem}

This theorem shows that the approximation error of the Galerkin
numerical solution for the gFEM depends directly on the accuracy of
the local approximation spaces in each patch.

\subsection{Low-Rank Local Solution Space}\label{sec:low_rank}

One reason for the success of gFEM is that the local numerical
solution space is approximately low-rank, meaning that $n_i$ has a
modest value for all $i$ in~\eqref{eqn:local_space}; see
\cite{BL11}. We review the relevant results in this section, and show
how to find $\Phi_{[i]}$.

Denote by $\omega_i^\ast$ an enlargement of the patch $\omega_i$,
\steveedit{that is, a set for which $\omega_i \subset \omega_i^\ast
  \subset \Omega$.} 
To simplify notation, we suppress subscripts $i$
from here on. We introduce a restriction operator:
\begin{equation*}
P: H_a(\omega^\ast)/\mathbb{R}\rightarrow H_a(\omega)/\mathbb{R}\,,
\end{equation*} 
where $H_a(\omega^\ast)$ is the collection of all $a$-harmonic
functions in $\omega^\ast$ and $H_a(\omega^\ast)/\mathbb{R}$
represents the quotient space of $H_a(\omega^\ast)$ with respect to the
constant function. (This modification is needed to make
$\|\cdot\|_{\mathcal{E}(\omega^\ast)}$ a norm, since an $a$-harmonic
function is defined only up to an additive constant.)
The operator $P$ is determined uniquely by $a(x)$ restricted in
$\omega^\ast$ and $\omega$. 
We denote its adjoint
operator by $P^\ast: H_a(\omega)/\mathbb{R}\rightarrow
H_a(\omega^\ast)/\mathbb{R}$. It is shown in \cite{BL11} that the
operator $P^\ast P$ is a compact, self-adjoint, nonnegative operator
on $H_a(\omega^\ast)/\mathbb{R}$.

\blueedit{
To derive an $n$-dimensional approximation of
$H_a(\omega^\ast)/\mathbb{R}$, we define as follows the \textit{Kolmogorov distance} of an arbitrary
$n$-dimensional function subspace $S_n \subset H_a(\omega)/\mathbb{R}$
to $H_a(\omega^\ast)/\mathbb{R}$ associated with their corresponding norm $\|\cdot\|_{\mathcal{E}(\omega)}$ and $\|\cdot\|_{\mathcal{E}(\omega^\ast)}$ respectively:
\begin{equation}\label{eqn:Kolmogorov_distance}
d(S_n,H_a(\omega^\ast)) = \sup_{\substack{u\in H_a(\omega^\ast)/\mathbb{R}\,, \\ \|u\|_{\mathcal{E}(\omega^\ast)}\leq 1}} \, \inf_{\xi \in S_n} \,\|Pu-\xi\|_{\mathcal{E}(\omega)}\,.
\end{equation}
(We omit the  norms  from the arguments of $d$, since they are clear from the context.)}
By considering all possible $S_n$, we can identify the optimal
approximation space $\Phi_n$ 
that achieves the infimum: 
\blueedit{
\begin{equation} \label{eq:def.Phin}
\Phi_n := \arg \inf_{S_n} d(S_n, H_a(\omega^\ast)) \,.
\end{equation}
We now define a distance measure between $\omega$ and $\omega^\ast$ as
follows: 
\[
d_n(\omega,\omega^\ast)=d(\Phi_n, H_a(\omega^\ast)) \,.
\]
}
The term $d_n(\omega,\omega^\ast)$ is the celebrated Kolmogorov
$n$-width of the compact operator $P$. It reflects how quickly
$a$-harmonic functions supported on $\omega^\ast$ lose their energies
when confined to $\omega$. According to~\cite{Pinkus}, the
\keedit{optimal approximation space} $\Phi_n$ and Kolmogorov
$n$-width can be found explicitly, in terms of the eigendecomposition
of $P^\ast P$ on $\omega^\ast$, which is
\begin{equation}\label{eqn:GEP_P}
P^\ast P \psi_i=\lambda_i \psi_i\,, \quad i=1,2,\dotsc,
\end{equation}
with $\lambda_i$ arranged in descending order \keedit{and
  $\{\psi_i,\ i=1,2,\ldots\}$ the corresponding eigenvectors, which
  are automatically orthonormal according to
  $\langle\cdot,\cdot\rangle_{\CalE(\omega^*)}$.} 
\steveedit{By defining
\begin{equation}\label{eqn:def_Psi}
\Psi_n := \Span\{\psi_1\,,\cdots\,,\psi_n\},
\end{equation}
the 
\keedit{optimal approximation space} is
\begin{equation} \label{eqn:def_Phi}
\Phi_n :=P\Psi_n = \Span\{\phi_1,\phi_2,\dotsc\,,\phi_n\}\,, \quad \mbox{with $\phi_i := P \psi_i$, $i=1,2,\dotsc,n$.}
\end{equation}
It follows from the definitions above that }
\begin{equation}\label{eqn:n_width}
d_n(\omega,\omega^\ast) = \sqrt{\lambda_{n+1}}\,.
\end{equation}
Note that $\psi_i$ are all supported in the enlarged domain
$\omega^\ast$, while $\phi_i$ are their confinements in $\omega$.
Almost-exponential decay of the Kolmogorov width with respect to $n$
was proved in \cite[Theorem~3.3]{BL11}, according to the following result. 
\begin{theorem}\label{thm:accuracy}
The accuracy $d_n(\omega,\omega^\ast)$ has nearly exponential decay
for $n$ sufficiently large: For any small $\varepsilon>0$, we
have
\[
d_n(\omega,\omega^\ast)\leq e^{-n^{\jledit{(d+1)^{-1}-\varepsilon}}}\,.
\]
It follows that for any function $u$ that is $a$-harmonic function in
the patch $\omega^\ast \subset \mathbb{R}^2$, we can find a function
$v\in\Phi_n$ for which
\[
  \|u-v\|_{\CalE(\omega)} \leq d_{n}(\omega,\omega^\ast)\|u\|_{\CalE(\omega^\ast)}
  \jledit{\leq \sim e^{-n^{1/3} - \veps}}\,.
\]
\end{theorem}

\begin{remark}
  Note that $d_n$ is the $(n+1)$-th singular value
  $\sqrt{\lambda_{n+1}}$ of $P$. Because of the fast decay of $d_n$
  with respect to $n$ indicated by Theorem \ref{thm:accuracy}, $P$ is
  an approximately-low-rank operator. It follows that almost all
  $a$-harmonic functions supported on $\omega^\ast$, when confined in
  $\omega$, look almost alike, and can be represented by a relatively
  small number of ``important'' modes.
\end{remark}

\begin{remark}
	We note that enlarging the domain for over-sampling is \jledit{a standard approach:}  In~\cite{HW97}, the boundary layer behavior confined in $\omega^\ast/\omega$ was studied and utilized for computation.
\end{remark}

\subsection{Computing the Local Solution Space}\label{sec:local_basis}

We describe here the computation of an approximation to $\Phi_n$ via
\steveedit{discretized versions of the objects} defined in the
previous subsection. 
\keedit{More specifically, we discretize the enlarged patch $\omega^\ast$ with a fine mesh, and collect all $a$-harmonic functions upon discretization.} To collect all $a$-harmonic functions, we would need to solve the system with elliptic operator \eqref{eq:elliptical} locally, with
all possible Dirichlet boundary conditions on $\partial\omega^\ast$. For
ease of presentation, here and in sequel, we assume that we choose a
piecewise-affine finite-element discretization of the patch for
computing the local $a$-harmonic functions. Then the discretized
$a$-harmonic functions are determined by their values on grid points
$\{y_1,y_2,\dotsc,y_{N_y}\}$ on the boundary of $\partial\omega^\ast$.
We proceed in three stages.
\paragraph*{\textbf{Stage A}}
\keedit{Construct the discrete $a$-harmonic function space $H_a(\omega^\ast)$ on the fine mesh }
via the functions $\chi_i$ obtained by solving the following system, for
$i=1,2,\dotsc,N_y$:
\begin{equation}\label{eqn:delta}
\begin{cases}
\mathcal{L}\chi_i = -\text{div}\left(a(x)\nabla \chi_i\right)= 0\,,\quad x\in\omega^\ast\\
\chi |_{\partial\omega^\ast} = \delta_i\,,\quad y_i\in\partial\omega^\ast,
\end{cases}\,
\end{equation}
where $\delta_i$ is the hat function that peaks at $y_i$ and equals 
zero at other grid points $y_j$, $j \ne i$. 
Recall that we have assumed a piecewise-affine finite-element
discretization of $\omega^{\ast}$. 

\paragraph*{\textbf{Stage B}}
%

  \keedit{Compute the eigenvalue problem~\eqref{eqn:GEP_P}} in the space spanned by $\{\chi_i\,,i=1,2,\dotsc,N_y\}$. Noting that
  \keedit{
  \begin{equation}\label{eqn:compute_psi}
  \langle P^\ast P\psi_i\,,\delta\rangle_{\CalE(\omega^*)} = \langle P\psi_i\,,P\delta\rangle_{\CalE(\omega)} =\langle \psi_i\,,\delta\rangle_{\CalE(\omega)}\,,\quad\forall\ \delta\in \text{span}\ \{\chi_1,\ldots,\chi_{N_y}\}\,,
  \end{equation}
  }
the weak formulation of \keedit{the eigenvalue problem~\eqref{eqn:GEP_P}, when confined in the discrete $a$-harmonic function space,} is given by
\begin{equation*}
\langle \psi_i\,,\chi \rangle_{\CalE(\omega)}= \lambda_i \langle \psi_i\,,\chi \rangle_{\CalE(\omega^*)}\,,\quad\forall \chi \in\Span\{\chi_1,\ldots,\chi_{N_y}\}\,. 
\end{equation*}
Expanding the eigenfunction $\psi_i$ in terms of $\chi_j$,
$j=1,2,\dotsc,N_y$, as follows:
\begin{equation} \label{eq:psii_ansatz}
\psi_i = \sum_{j}c^{(i)}_{j}\chi_j,
\end{equation}
we obtain the following equation for the coefficient vector
$\vec{c}^{(i)}$:
\begin{equation*}
\sum_j c^{(i)}_{j}\langle \chi_j\,,\chi_k\rangle_{\CalE(\omega)}= \lambda_i \sum_jc^{(i)}_{j}\langle \chi_j\,,\chi_k\rangle_{\CalE(\omega^*)}\,.
\end{equation*}
This system can be written as \keedit{a genearlized eigenvalue problem, as follows:}
\begin{equation}\label{eqn:GEP}
\begin{gathered}
\Smat\vec{c}^{(i)} =
\lambda_i\Smat^\ast\vec{c}^{(i)}\,,\\ \text{with} \; \Smat_{mn} =
\langle \chi_m\,,\chi_n\rangle_{\CalE(\omega)},
\; \text{and} \quad \Smat^\ast_{mn}= \langle
\chi_m\,,\chi_n\rangle_{\CalE(\omega^*)}\,, \quad m,n = 1,2,\dotsc,N_y.
\end{gathered}
\end{equation}

This generalized eigenvalue problem can be solved for $\lambda_i$,
$i=1,2,\dotsc,N_y$, arranged in descending order, and their associated
eigenfunctions $\psi_i$, $i=1,2,\dotsc,N_y$ defined from
\eqref{eq:psii_ansatz} using the generalized eigenvectors
$\vec{c}^{(i)}$ from \eqref{eqn:GEP}. 
Choose index $n$ to satisfy $\lambda_{n+1}<\TOL<\lambda_n$, where
$\TOL$ is a given error tolerance.
\paragraph*{\textbf{Stage C}}
Obtain $\Phi_n$ by substituting the functions $\psi_i$,
$i=1,2,\dotsc,n$ calculated in Stage B into \eqref{eqn:def_Phi}.

\section{Randomized Sampling Methods for Local Bases}\label{sec:methods}

In this section we propose a class of random sampling methods to
construct local basis functions efficiently. As seen in
Section~\ref{sec:local_basis}, finding the optimal basis functions
amounts to solving the generalized eigenvalue problem
in~\eqref{eqn:GEP}. The main cost comes not from performing the
eigenvalue decomposition, but rather from computing the $a$-harmonic
functions $\chi_i$, which are used to construct the matrices $\Smat$
and $\Smat^\ast$ in \eqref{eqn:GEP}. As shown in
Section~\ref{sec:low_rank}, the eigenvalues decay almost
exponentially, indicating that \keedit{only a limited number of
  local modes is needed to represent the whole solution space
  well}. This low-rank structure motivates us to consider randomized
sampling techniques.

Randomized algorithms have been highly successful in compressed
sensing, where they are used to extract low-rank structure efficiently
from data. The Johnson-Lindenstrauss lemma
\cite{johnson1984extensions} suggests that structure in high
dimensional data points is largely preserved when projected onto
random lower-dimensional spaces. The randomized SVD (rSVD) algorithm
uses this idea to captures the principal components of a large matrix
by random projection of its row and column spaces into smaller
subspaces; see \cite{Tropp_rSVD} for a review. In the current
numerical PDE context, knowing that the local solution space is
essentially low-rank, we seek to adopt the random sampling idea to
generate local approximate solution spaces efficiently. 

Randomized SVD cannot be applied directly in our context, as we
discuss in Section~\ref{sec:svd}. We propose instead a method based on
Galerkin approximation of the generalized eigenvalue problem on a
small subspace. One immediate difficulty is that an arbitrarily given
random function is not necessarily $a$-harmonic. Thus, our method
first generates a random collection of functions and projects them
onto the $a$-harmonic function space, and then solves the generalized
eigenvalue problem~\eqref{eqn:GEP} on the subspace to find the optimal
basis functions. A detailed description of our approach is shown in
Algorithm~\ref{alg:1}.
\begin{algorithm}
\begin{algorithmic}
\STATE {\bf Stage 1:} Randomly generate a collection of $N_r$
$a$-harmonic functions.
\begin{ALC@g}
\STATE {\bf Stage 1-A:} Randomly pick functions $\{\xi_k \, : \, k=1,2,\dotsc,N_r\}$ supported on $\omega^\ast$;
\STATE {\bf Stage 1-B:} For each $k=1,2,\dotsc,N_r$, project $\xi_{k}$ onto the $a$-harmonic
  function space to obtain $\gamma_k$;
\end{ALC@g}
\STATE {\bf Stage 2:} Solve the generalized eigenvalue problem to determine leading modes.
\begin{ALC@g}
\STATE {\bf Stage 2-A:} Define:
\begin{equation} \label{eqn:SSstar}
  \Smat_{\gamma,mn} := \langle \gamma_m\,,\gamma_n\rangle_{\CalE(\omega)}\,,\quad\Smat^\ast_{\gamma,mn} := \langle \gamma_m\,,\gamma_n\rangle_{\CalE(\omega^*)}\,, \quad
  m,n=1,2,\dotsc,N_r,
\end{equation}
and solve the associated generalized eigenvalue problem:
\begin{equation}\label{eqn:ran_GEP}
\Smat_{\gamma}\vec{v}^\gamma = \lambda^\gamma\Smat_{\gamma}^{\ast}\vec{v}^\gamma\,,
\end{equation}
with $(\lambda^{\gamma}_{j}\,,\vec{v}^{\gamma}_{j})$ denoting the
$j$-th eigen-pairs, such that \steveedit{$\lambda^{\gamma}_1 \ge
  \lambda^{\gamma}_2 \ge \cdots \ge \lambda^{\gamma}_{N_r} \ge 0$;}
\STATE {\bf Stage 2-B:} \steveedit{Choose $n$ such that
  $\lambda^\gamma_1\geq\cdots\geq\lambda^\gamma_n\geq\text{TOL}>\lambda^\gamma_{n+1}$
  (where $\text{TOL}$ is a preset tolerance) and collect the first $n$
  eigenfunctions to use as the local basis functions:}
\begin{equation}
\label{eqn:RanOptimalSpace}
\Phi_n^{r}=\Span\{\phi^r_1,\phi^r_2,\dotsc,\phi^r_n\} = \Span\{P\psi_j^r\,,j=1,2,\dotsc,n\} \,,
\end{equation}
where $\psi_j^r = \sum_{k}\vec{v}^\gamma_{j,k}\gamma_{k}$.
\end{ALC@g}
\end{algorithmic}
\caption{Determining Optimal Local Bases\label{alg:1}}
\end{algorithm}

Note that the steps in Stage 2 of Algorithm~\ref{alg:1} are parallel to
those of Section~\ref{sec:local_basis}, but only a small number of
functions $\gamma_k$ is used in the generalized eigenvalue problem,
rather that the whole list of $a$-harmonic functions (i.e.,
$N_r \ll N_y$). We therefore save significant computation in preparing the
$a$-harmonic function space, in assembling the $\Smat$ and
$\Smat^{\ast}$ matrices, and in solving the generalized eigenvalue
decomposition.

The key is to use the random sampling strategy in Stage 1 of
Algorithm~\ref{alg:1} to generate an effective small subspace for the
generalized eigenvalue problem. This aspect of the algorithm will be the focus of the rest
of this section.

\subsection{$a$-Harmonic Projection}
\label{sec:harmonic}

Let us first discuss the $a$-harmonic projection of a given function $\xi$ supported on $\omega^\ast$. This problem can be formulated as a
PDE-constrained optimization problem:
\begin{equation}\label{eqn:harmonic_projection}
\min_{\gamma} \frac{1}{2}\|\gamma-\xi\|^2_{L^2(\omega^\ast)}\quad \text{subject to}\quad \mathcal{L}\gamma=0\,,
\end{equation}
where $\mathcal{L} = -\divop a \nabla $ is the elliptic operator defined in
\eqref{eqn:elliptic}. \blueedit{The Lagrangian function for 
\eqref{eqn:harmonic_projection} is as follows:
\begin{equation}
\label{eqn:lagrangian}
 F(\gamma,\mu) := \frac{1}{2}\|\gamma-\xi\|^2_{L^2(\omega^\ast)} - \langle\mu\,,\mathcal{L}\gamma\rangle_{L^2(\omega^*)}\,,
\end{equation}
where $\mu$ is a Lagrange multiplier.} In the discrete setting, we
form a grid $\{x_i\}$ over $\omega^\ast$ and denote by $\zeta_i$ the hat
function centered at grid point $x_i$. (Recall that we have assumed
piecewise-affine finite element discretization.) The Lagrangian function  for the corresponding discretized optimization problem is 
\begin{equation}\label{eqn:harmonic_projection_dis}
F(\gamma,\mu) = \frac{1}{2}(\gamma^i-\xi^i)^\top(\gamma^i-\xi^i)-\mu^\top \Amat^{ii}\gamma^i-\mu^\top \Amat^{ib}\gamma^b\,,
\end{equation}
where the superindices $i$ and $b$ stand for interior and boundary grids, respectively, and $\Amat$ is the stiffness matrix whose $(m,n)$ element is
\begin{equation*}
\Amat_{mn} = \langle a\nabla \zeta_m\,, \nabla \zeta_n\rangle_{L^2(\omega^*)}\,.
\end{equation*}
In the discrete setting, $\mu$ is a vector of the same length as $\gamma^i$ (the number of grid points in the interior). Note that in the translation to the discrete setting, we represent $\mathcal{L}\gamma=0$  by
$\Amat\gamma = 0$,
which leads to
\begin{equation*}
\Amat^{ii}\gamma^i + \Amat^{ib}\gamma^b = 0\,.
\end{equation*}
Here $\Amat^{ii}$ is the stiffness matrix confined in the interior,
and $\Amat^{ib}$ is the part of the stiffness matrix generated by
taking the inner product of the interior basis functions and the
boundary basis functions. To solve the minimization problem, we take
the partial derivatives of~\eqref{eqn:harmonic_projection_dis} with
respect to $\gamma$ and $\mu$ and set them equal to zero, as follows:
\begin{alignat*}{2}
\nabla_{\gamma^i} F&=\gamma^i-\xi^i-\Amat^{ii\top}\mu && =0\,,\\
\nabla_{\gamma^b} F&=\Amat^{ib\top}\mu && = 0\,,\\
\nabla_\mu F& = \Amat^{ii}\gamma^i+\Amat^{ib}\gamma^b &&= 0\,.
\end{alignat*}
Some manipulation yields the following systems for $\gamma^b$ and $\gamma^i$:
\begin{equation*}
\Amat^{ib\top}\left(\Amat^{ii}\right)^{-2}\Amat^{ib}\gamma^b = -\Amat^{ib\top} \left(\Amat^{ii}\right)^{-1}\xi^i,\quad \gamma^i=-\left(\Amat^{ii}\right)^{-1}\Amat^{ib}\gamma^b\,.
\end{equation*}
The solution to this system gives the solution
of~\eqref{eqn:harmonic_projection} in the discrete setting.
\blueedit{ Recall that $\gamma^b$ is a vector containing only the
    boundary conditions for the solution, and thus the computation is
    rather cheap, given that the matrix
    $\Amat^{ib\top}\left(\Amat^{ii}\right)^{-2}\Amat^{ib}$ can be
    prepared ahead of time. Computing $\gamma^i$ using $\gamma^b$
    amounts to numerically solving a finite element problem confined
    in a small domain $\omega^\ast$, and thus the numerical cost is
    the same as preparing an $a$-harmonic function.  }

\subsection{Random sampling strategies}\label{sec:Random_strat}

We have many possible choices for the random functions functions $\xi_k$,
$k=1,2,\dotsc,N_r$ in Stage 1-A of \ref{alg:1}. Here we list several natural choices.
\begin{itemize}
\item[1.]\emph{Interior $\delta$-function.} Choose a random grid point in
  $\omega$ and set $\xi(x) = 1$ at this grid point, and zero at all
  other grid points. That is, $\xi$ is the hat function associated with
  the grid point $x$.
\item[2.]\emph{Interior i.i.d.~function.} Choose the value of $\xi$ at each
  grid point in $\omega$ independently from a standard normal Gaussian distribution. The values of $\xi$ at grid points in $\omega^\ast\backslash\omega$ are set to $0$.
\item[3.]\emph{Full-domain i.i.d.~function.} The same as in 2, except that
  the values of $\xi$ at the grid points in
  $\omega^{\ast}\backslash\omega$ are also chosen as Gaussian random
  variables.
\item[4.]\emph{Random Gaussian.} Choose a random grid point
  $x_0\in\omega$ and set $\xi(x)=e^{-\frac{(x-x_0)^2}{2}}$ at all
  grid points $x \in \omega^{\ast}$. 
\end{itemize}
We aim to select basis functions (through Stage 2) that are
associated with the largest eigenvalues, 
so that the Kolmogorov $n$-width can be small~\eqref{eqn:n_width}. 
Thus, we hope that in  Stage 1, the chosen functions $\xi_k$ provide large eigenvalues $\lambda_i$ in \eqref{eqn:GEP}. A large value of
$\lambda$ indicates that a large portion of the energy is maintained
in $\omega$, with only a small amount coming from the buffer region
$\omega^\ast\backslash\omega$. It therefore suggests to choose
functions $\xi_k$ with most of their variations inside
$\omega$. However, the projection to $a$-harmonic space step makes the
locality of the resulting functions hard to predict.  In
Section~\ref{sec:criteria}, we propose and analyze a criterion for
the performance of the random sampling schemes. In particular, we
compare the four  choices listed above.

\smallskip 
\blueedit{ We mention here that a list of $a$-harmonic functions
    could be obtained through a different route: one can prepare
    boundary conditions and compute local $a$-harmonic function inside
    $\omega^\ast$ with the pre-assigned boundary. There are various
    ways to prepare boundary conditions, including the following. }


\begin{enumerate} 
	
\item[5.] \emph{Random \keedit{i.i.d.} boundary sampling.}
In~\cite{calo2016randomized}, the authors proposed to obtain a list of random
$a$-harmonic functions by computing the local elliptic equation with
i.i.d.~random Dirichlet boundary conditions. Assuming there are $N_y$
grid points on the boundary $\partial\omega^\ast$, we define $g$ to be
a vector of length $N_y$ with i.i.d.~random variables for each
component. We then define $\gamma$ by solving
\begin{equation}\label{eqn:ran_BVP}
\begin{cases}
\mathcal{L}\gamma= 0\,,\quad x\in\omega^\ast,\\
\gamma |_{\partial\omega^\ast} = g.
\end{cases}
\end{equation}
This process is repeated  $N_r$ times to obtain a set of $N_r$ random
$a$-harmonic functions $\{\gamma_k \, : \, k=1,2,\dotsc,N_r\}$.

\item[6.] \emph{Randomized boundary sampling with exponential covariance.}
A technique in which the Dirichlet boundary conditions are chosen to
be random Gaussian variables with a specified covariance matrix is
described in \cite{Lipton16}. This matrix is assumed to be an
exponential function,
that is,
\begin{equation}
\Cov(y_i,y_j) = \exp(-|y_i-y_j|/\sigma)\,.
\end{equation}
The first few modes of a Karhunen-Lo\'{e}ve expansion are used to
construct a boundary condition in~\eqref{eqn:ran_BVP}, with which
basis functions are computed. Although a justification for this
approach is not provided, numerical computations show that it is more 
efficient than the i.i.d.~random boundary sampling. 

\blueedit{
\item[7.] \emph{Smooth boundary sampling.} 
  	I.i.d. random Dirichlet
    boundary conditions typically yield solutions that oscillate a lot
    near the boundary, and thus have sharp boundary layers. To
    eliminate this effect, one can use a Gaussian kernel to smooth out
    the boundary profile. In particular, the i.i.d. random sample can
    be convolved with a Gaussian function
    $\frac{1}{\sqrt{2\pi}\sigma}e^{-x^2/2\sigma^2}$ to obtain a
    smoother boundary condition.
%
}
\end{enumerate}
\blueedit{ We note that Strategies 5 and 6 above were proposed
  in~\cite{calo2016randomized} and~\cite{Lipton16}
  respectively. However, in~\cite{calo2016randomized}, the
  post-processing for basis selection was conducted using the
  generalized eigenvalue problem of the stiffness and mass matrix
  instead of Equation \eqref{eqn:GEP_P}, and thus there is no guarantee in the exponential decay.  }

\subsection{Connection with Randomized SVD}\label{sec:svd}

We briefly address the connection between the random sampling method
we propose in this paper and the well-known randomized SVD (rSVD)
algorithm. Although rSVD cannot be used directly in our problem, it
serves as a motivation for our randomized sampling strategies.


The randomized SVD algorithm, studied thoroughly in~\cite{Tropp_rSVD},
speeds up the computation of the SVD of a matrix when the matrix is large and
approximately  low rank. With high probability, the singular vector
structure is largely preserved when the matrix is projected onto a
random subspace. Specifically, for a random matrix $\Rmat$ with a
small number of columns (the number depending on the rank of $\Amat$), it is
proved in \cite{Tropp_rSVD} that if we obtain $\Qmat$ from the QR
factorization of $\Amat\Rmat$, we have
\begin{equation} \label{eq:rsvd.1}
\|\Amat - \Qmat\Qmat^\top\Amat\|_2\ll \| \Amat\|_2\,.
\end{equation}
This bound implies that any vector in the range space of $\Amat$ can
be well approximated by its projection into the space spanned by
$\Qmat$. For example, if $\vec{u} = \Amat\vec{v}$, we have from
\eqref{eq:rsvd.1} that
\begin{equation} \label{eq:rsvd.2}
\|\vec{u} - \Qmat\Qmat^\top\vec{u}\|\ll\|\vec{u}\|\,.
\end{equation}
We note that $\Qmat$ and
$\Amat\Rmat$ span the same column space, but $\Qmat$ is easier to
work with and better conditioned, because its columns are
orthonormal. Equivalent to \eqref{eq:rsvd.2}, we can also say that any
$\vec{u}$ in the image of $\Amat$ can be approximated well using a
linear combination of the columns of $\Amat\Rmat$.

To see the connection between rSVD and our problem, we first write the
generalized eigenvalue problem \eqref{eqn:GEP} in a SVD
form. Recall the definitions \eqref{eqn:SSstar} of $\Smat$ and
$\Smat^\ast$:
\begin{equation*}
\Smat_{mn} = \int_{\omega}a(x)\nabla \chi_m(x)\cdot\nabla \chi_n(x)\rd{x}\,,\quad \Smat^\ast_{mn} = \int_{\omega^\ast}a(x)\nabla \chi_m(x)\cdot\nabla \chi_n(x)\rd{x}\,,
\end{equation*}
and define
\begin{equation}\label{eqn:def_phi}
\Phi^\ast = \left[\sqrt{a}\nabla \chi_1,\sqrt{a}\nabla \chi_2,\dotsc,\sqrt{a}\nabla \chi_{N_y}\right]\,,\quad \Phi = \Phi^\ast|_\omega\,.
\end{equation}
Since $\Smat = \Phi^\top\Phi$ and $\Smat^\ast =
\Phi^{\ast\top}\Phi^\ast$, the generalized eigenvalue
problem~\eqref{eqn:GEP} can be written as follows:
\begin{equation} \label{eq:GEP.2}
\Phi^\top\Phi\vec{c} = \lambda\Phi^{\ast\top}\Phi^\ast\vec{c}\,.
\end{equation}
We write the  QR factorization for $\Phi^\ast$ as follows: 
\begin{equation*}
\Phi^\ast = \Qmat_{\Phi^*}  \Rmat_{\Phi^*} \,,
\end{equation*}
and denote $\vec{d} = \Rmat_{\Phi^*} \vec{c}$. By substituting into
\eqref{eq:GEP.2}, we obtain \keedit{
\begin{equation*}
\left(\Phi \Rmat_{\Phi^*}^{-1}\right)^\top\left(\Phi\Rmat_{\Phi^*}^{-1}\right)\vec{d} = \lambda \vec{d}\,,
\end{equation*}
meaning that $\left(\sqrt\lambda,\vec{d}  \right)$ forms a singular value pair of the matrix $\Phi R^{-1}_{\Phi^\ast}$. }

According to the rSVD argument, the leading singular vectors of
$\Phi\Rmat^{-1}_{\Phi^*}$ are captured by those of
\begin{equation}\label{eqn:phi_new}
\Phi\Rmat_{\Phi^*}^{-1} \Rmat\,,
\end{equation}
where $\Rmat$ is a matrix whose entries are i.i.d Gaussian random
variables.  Specifically, with high probability, the leading singular
values of $\Phi\Rmat_{\Phi^*}^{-1}\Rmat$ are almost the same
as those of $\Phi\Rmat_{\Phi^*}^{-1}$, and the column space
spanned by~\eqref{eqn:phi_new} largely covers the image of
$\Phi\Rmat_{\Phi^*}^{-1}$, as in~\eqref{eqn:def_phi}. 

We now interpret $\Phi\Rmat^{-1}_{\Phi^*}\Rmat$ from the
viewpoint of PDEs. Decomposing $\Rmat_{\Phi^*}^{-1}\Rmat$ into
columns as follows:
\begin{equation}\label{eqn:ideal_r}
\Rmat_{\Phi^*}^{-1} \Rmat = [r_1,r_2,\dotsc]\,,\quad\text{with}\quad 
r_k = [r_{k1},r_{k2},\dotsc]^\top\,,
\end{equation}
and denoting $\gamma_k = \sum_{j}r_{kj}\chi_j$,
we have from \eqref{eqn:def_phi} that 
\begin{equation*}
\Phi r_k = \sqrt{a}\nabla \biggl(\sum_{j}r_{kj}\chi_j\biggr)\doteq \sqrt{a}\nabla \gamma_k\,.
\end{equation*}
Numerically, this corresponds to solving the following system for
$\gamma_k$:
\begin{equation}\label{eqn:elliptic_gamma}
\begin{cases}
\mathcal{L}\gamma_k = -\text{div}\left(a(x)\nabla \gamma_k \right) = 0\,,\quad x\in\omega^\ast,\\
\gamma_k|_{\partial\omega^\ast} = \sum_jr_{kj}\delta_{y_j}.
\end{cases}
\end{equation}
It is apparent from this equation that to obtain
$\Phi\Rmat^{-1}_{\Phi^*}\Rmat$, we do not need to compute
all functions $\chi_j$, $j=1,2,\dotsc,N_y$ and use them to construct
$\gamma_k$. Rather, we can compute $\gamma_k$ directly by solving the
elliptic equation with random boundary conditions given by $r_{kj}$,
$j=1,2,\dotsc,N_y$. The cost of this procedure is proportional to
$N_r$, which is much less than $N_y$.

Unfortunately, this procedure is difficult to implement in a manner
that accords with the rSVD theory.  $\Rmat$ is constructed using
i.i.d.~Gaussian random variables, but $\Rmat^{-1}_{\Phi^*}$ is unknown
ahead of time, so the distribution of $r_k$ defined in
\eqref{eqn:ideal_r} is unknown. The theory here suggests that there
\keedit{exists some} random sampling strategy that achieves the
accuracy and efficiency that characterize rSVD, but it does not
provide such a strategy.

\section{Efficiency of Various Random Sampling Methods}\label{sec:criteria}

As discussed in Section~\ref{sec:methods}, given multiple ways to
choose the random samples in Stage 1 of Algorithm~\ref{alg:1}, it is
natural to ask which one is better, and how to predetermine the
approximation accuracy. We answer these questions in this section.

The key requirement is that Algorithm~\ref{alg:1} should capture the
high-energy modes of \eqref{eqn:GEP_P}, the modes that correspond to
the highest values of $\lambda_i$.
We start with a simple example in
Section~\ref{sec:1D} that finds the relationship between the energy
captured by a certain single mode, and the angle that that mode makes
with the highest energy mode.
The argument used can be easily applied to the case with multiple modes and the link towards the Kolmogorov distance will be shown in Section~\ref{sec:highD}. We will discuss the situation in the general setting with plain linear algebra, and its relevance to local PDE basis construction is outlined in Section~\ref{sec:basis_energy}.

\subsection{A One-Mode Example}\label{sec:1D}

Suppose we are working in a three-dimensional space, with symmetric positive definite matrices $\Amat$ and $\Bmat$ and generalized eigenvectors $x_1$,
$x_2$, and $x_3$ such that
\begin{equation} \label{eq:angle.delta}
\langle x_i\,,x_j\rangle_\Bmat  = x_i^\top\Bmat x_i = \delta_{ij}\,,\quad \langle x_i\,,x_j\rangle_\Amat  = x_i^\top\Amat x_i = \delta_{ij}\lambda_i\,,
\end{equation}
for generalized eigenvalues $\lambda_1>\lambda_2>\lambda_3$. We thus have
\begin{equation*}
\Amat x_i = \lambda_i\Bmat x_i, \quad i=1,2,3.
\end{equation*}
\blueedit{ Suppose we have some one-dimensional space $\Xcal$ spanned by
  a vector $x$, and we intend to use it to as an approximation of the
  space $\Xcal_1$ spanned by the leading eigenvector $x_1$. The energy
  of $\Xcal$ is 
\begin{equation}\label{eqn:energy_1}
E(\Xcal) = \frac{x^\top\Amat x}{x^\top\Bmat x}\,,
\end{equation}
and the angle between the spaces $\Xcal$ and $\Xcal_1$ is defined by}
\blueedit{
\begin{equation}\label{eqn:angle_1}
d(\Xcal,\Xcal_1) = \max_{|\beta|\leq 1}\min_{\alpha}\|\alpha x - \beta x_1\|_\Amat\,.
\end{equation}
}


We have the following result (which generalizes easily to dimension greater than $3$).
\begin{proposition}\label{thm:1d_example}
The angle \eqref{eqn:angle_1} is bounded in terms of
the energy \eqref{eqn:energy_1} as follows:
\blueedit{
	\begin{equation}\label{eqn:angle_bound_1}
	d(\Xcal,\Xcal_1)\leq\sqrt{\frac{\lambda_1\lambda_2\left(\lambda_1-E(\Xcal)\right)}{(\lambda_1-\lambda_2)E(\Xcal)}}\,.
	\end{equation}   }
\end{proposition}
\begin{proof}
The proof is simple algebra. As $\{x_1,x_2,x_3\}$ span the entire
space and are $\Bmat$-orthogonal, we have
\begin{equation} \label{eq:xw}
x = w_1x_1+w_2x_2+w_3x_3\,,
\end{equation}
with $w_i = x^\top\Bmat x_i$, $i=1,2,3$. According to the definition
of the angle, one can reduce the problem by setting $\beta=1$ and
$\sum_iw_i^2=1$, so that $\|x\|_{\Bmat}=1$ in
\eqref{eqn:angle_1}. (With these normalizations, we have from
\eqref{eq:angle.delta} and \eqref{eqn:energy_1} that $E(\Xcal) = x^TAx =
\lambda_1 w_1^2 + \lambda_2 w_2^2 + \lambda_3 w_3^2$.)
We thus have
\begin{align*}
d(\Xcal,\Xcal_1)^2 &= \min_{\alpha}\|\alpha x - x_1\|^2_\Amat\\
& = \min_{\alpha}\|(\alpha w_1-1) x_1 + \alpha w_2x_2 + \alpha w_3x_3\|^2_\Amat\\
& = \min_\alpha \left((\alpha w_1-1)^2\lambda_1+ \alpha^2 w^2_2 \lambda_2 + \alpha^2w^2_3\lambda_3\right).
\end{align*}
The minimum is achieved at $\alpha = w_1\lambda_1/E(\Xcal)$,
with the minimized angle being
\begin{equation} \label{eq:ss1}
\left(\angle(x,x_1)\right)^2 = \frac{E(\Xcal) - w_1^2\lambda_1}{E(\Xcal)}\lambda_1\,.
\end{equation}
To bound the numerator in \eqref{eq:ss1} we observe that
\begin{equation*}
E(\Xcal) - w_1^2\lambda_1 =  w_2^2\lambda_2 + w_3^2\lambda_3 \leq w_2^2 \lambda_2 + w_3^2 \lambda_2 = (1-w_1^2)\lambda_2,
\end{equation*}
and moreover
\[
E(\Xcal) \leq w_1^2\lambda_1 + (1-w_1^2)\lambda_2
\;\Rightarrow\;
\lambda_1 - E(\Xcal) \ge (1-w_1^2)(\lambda_1-\lambda_2)
\;\Rightarrow\;
1-w_1^2\leq\frac{\lambda_1-E(\Xcal)}{\lambda_1-\lambda_2}\,.
\]
By combining these last two bounds, we obtain
\[
E(\Xcal) - w_1^2 \lambda_1 \le \lambda_2 \frac{\lambda_1-E(\Xcal)}{\lambda_1-\lambda_2}.
\]
By substituting this bound into \eqref{eq:ss1}, we obtain
\eqref{eqn:angle_bound_1}.
\end{proof}

Note that the bound \eqref{eqn:angle_bound_1} decreases to zero as
$\lambda_1 - E(\Xcal) \downarrow 0$.

  According to~\eqref{eqn:angle_bound_1}, a larger gap in the
spectrum between $\lambda_1$ and $\lambda_2$ yields a tighter bound,
thus better control over the angle.  The theorem indicates 
that the ``energy" is the quantity that measures how well the randomly
given vector $x$ captures the first mode, and thus serves as the
criterion for the quality of the approximation.

\subsection{Higher-Dimensional Criteria}\label{sec:highD}

In this section, we seek the counterpart in higher dimensional space of the previous result. Suppose now that the two symmetric positive definite matrices
$\Amat$ and $\Bmat$ are $n\times n$, and their generalized eigenpairs
$(\lambda_i,x_i)$ satisfy the following conditions:
\begin{equation} \label{eq:ss4}
\langle x_i\,,x_j\rangle_\Bmat  = \delta_{ij}\,,\quad \langle x_i\,,x_j\rangle_\Amat  = \delta_{ij}\lambda_i\,,
\end{equation}
so that
\begin{equation*}
\Amat x_i = \lambda_i\Bmat x_i\,,\quad\text{with}\quad \lambda_1 \ge \cdots \ge \lambda_k>\lambda_{k+1} \ge \cdots \ge \lambda_n>0\,,
\end{equation*}
that is,
\begin{equation} \label{eq:sw65}
\Amat \Xmat = \Bmat \Xmat \Lambda\,,\quad\text{with}\quad\Lambda = \text{diag} (\lambda_1,\lambda_2,\dotsc,\lambda_n)\,.
\end{equation}



\blueedit{ Suppose we are trying to recover the optimal
  $k$-dimensional space $\Xcal^h:=$span$\{\Xmat^h\}$, where
  $\Xmat^h=[x_1,x_2,\ldots,x_k]$ collects the first $k$
  eigenfunctions. Define $\Xcal^l:=$span$\{\Xmat^l\} $, where $\Xmat^l
  = [x_{k+1},\ldots,x_n]$ collects the remaining modes. Denoting by
  $\Ycal$ our proposed approximation space to $\Xcal^h$, we seek a
  quantity that measures how well the proposal space $\Ycal$
  approximates the optimal space $\Xcal^h$. In particular, we will
  show below that the ``angle" between the proposal $\Ycal$ and the
  to-be-recovered space $\Xcal^h$ relies on the ``energy" of $\Ycal$.
}



\blueedit{
\begin{definition}[Energy of a space $\Zcal$]\label{def:energy}
	For any given $k$-dimensional space $\Zcal$, define
	$\Zmat\in\mathbb{R}^{n\times k}$ to be a matrix whose columns form a $\Bmat$-orthonormal basis of $\Zcal$
	(obtained through performing Gram-Schmidt with
	$\Bmat$-inner product). Then the energy of $\Zcal$ is defined as:
	\begin{equation}\label{eqn:energy}
	E(\Zcal):=\frac{\Tr(\Zmat^\top \Amat\Zmat)}{\Tr (\Zmat^\top \Bmat\Zmat)}\,.
	\end{equation}
\end{definition}
We note that this is a natural extension of energy defined
in~\eqref{eqn:energy_1}, and it is well-defined, in the sense that the
energy term~\eqref{eqn:energy} depends solely on the space $\Zcal$
rather than the basis $\Zmat$, as shown in Appendix
\ref{sec:appendixB}. }

%

\blueedit{We now generalize the angle \eqref{eqn:angle_1} and define the
  Kolmogorov distance from space $\Ycal$ to the optimal space
  $\Xcal^h$, with norms $\|\cdot \|_\Amat$ and $\|\cdot \|_\Bmat$,
  respectively.}
\blueedit{
\begin{definition}[Angle between spaces]
	Define the Kolmogorov distance from $\Ycal$ and the optimal subspace
	$\Xcal^h$ as follows:
	\begin{equation}\label{eqn:distance}
	d(\Ycal,\Xcal^h)=\sup_{\substack{z\in \Xcal^h ,\\ \|z\|_\Bmat\leq 1}} \, \inf_{y\in \Ycal} \, \|z-y\|_\Amat\,.
	\end{equation}
\end{definition}
} \keedit{Notice that $d(\Ycal,\Xcal^h)$ is a discrete version of
  \eqref{eqn:Kolmogorov_distance}, we don't have operator $P$ in
  \eqref{eqn:distance} since it is implicit in the norm
  $\|\cdot\|_\Amat$.}

\steveedit{Similar to the previous section, we show} that
$E(\Ycal)$ is related to $d(\Ycal,\Xcal^h)$.
\blueedit{
In Definition~\ref{def:energy} the energy $E(\Zcal)$ is defined for a $\Bmat$-orthonormal basis $\Zmat$, and for consistency we
assume $\Ymat\in\mathbb{R}^{n\times k}$ collects a $\Bmat$-orthonormal basis of space $\Ycal$. 
}
Since $\Xmat$ spans the entire space, we can express $\Ymat$ as follows:
\begin{equation} \label{eq:ss3}
\Ymat=\Xmat  \Cmat = \Xmat^h \Cmat^h + \Xmat^l \Cmat^l 
\end{equation}
where $\Cmat\in\mathbb{R}^{n\times k}$. 
 \steveedit{The columns of $\Cmat$ are orthonormal} because from
 \eqref{eq:ss4} and the definition of $\Ymat$
 , we have
\begin{equation} \label{eq:Corth}
\Ymat = \Xmat \Cmat \; \Rightarrow \; \mathbbm{1} = \Ymat^\top \Bmat
\Ymat = \Cmat^\top \Xmat^\top \Bmat \Xmat \Cmat = \Cmat^\top \Cmat.
\end{equation}
We denote by $\Cmat^h$ the upper $\mathbb{R}^{k\times k}$ portion of
$\Cmat$, and by $\Cmat^l$ the lower $\mathbb{R}^{(n-k)\times k}$
portion. Denoting the elements of $\Cmat$ by $c_{ji}$, we have
\begin{equation} \label{eq:def.cji}
\Cmat^h = [c_{ji}]_{j=1,2,\dotsc,k; \, i=1,2,\dotsc,k}, \quad
\Cmat^l = [c_{ji}]_{j=k+1,k+2,\dotsc,n; \, i=1,2,\dotsc,k}.
\end{equation}
By orthonormality of $\Cmat$, it follows that
\[
\sum_{j=1}^k c_{ji}^2 + \sum_{j=k+1}^n c_{ji}^2 = 1, \quad i=1,2,\dotsc,k.
\]
and thus
\begin{equation} \label{eq:Csum}
\left[ C^{l\top} C^l \right]_{ii} = 1 - \left[ C^{h\top} C^h \right]_{ii} =
1-\sum_{j=1}^k c_{ji}^2,
\quad i=1,2,\dotsc,k.
\end{equation}

\begin{lemma}\label{lemma:trace_CC}
The trace of $\Cmat^{l\top}\Cmat^l$ is bounded by energy difference between the optimal space \keedit{$\Xcal^h$} and the proposed space \keedit{$\Ycal$}
\begin{equation} \label{eq:ss7}
\Tr(\Cmat^{l\top} \Cmat^l) \leq \frac{k\left(E(\Xcal^h)-E(\Ycal)\right)}{\lambda_k-\lambda_{k+1}}\,.
\end{equation} 
Furthermore, $\Cmat^h$ is invertible if 
\begin{equation} \label{eq:sw9}
E(\Xcal^h)-E(\Ycal) < \frac{\lambda_k -\lambda_{k+1}}{k}\,.
\end{equation}
\end{lemma}
\begin{proof}
\steveedit{ We have from \eqref{eq:Corth} that}
\begin{equation}\label{eqn:Yortho}
\Cmat^{h\top} \Cmat^h+\Cmat^{l\top}\Cmat^l =\mathbbm{1}\,.
\end{equation}
Since both $\Xmat^h$ and $\Ymat$ are $\Bmat$-orthonormal and have $k$
columns, we have:
\begin{equation*}
\Tr \left(\Xmat^{h\top}\Bmat \Xmat^h \right)=\Tr \left(\Ymat^\top\Bmat \Ymat \right)=k\,.
\end{equation*}
By substituting $\Xmat^h$ and $\Ymat$ into the definition of
energy~\eqref{eqn:energy}, we have
\begin{equation*}
k\left(E(\Xcal^h)-E(\Ycal)\right)=\Tr \left(\Xmat^{h\top}\Amat\Xmat^h-\Ymat^\top
\Amat\Ymat \right).
\end{equation*}
By substituting for $\Ymat$ from \eqref{eq:ss3}, and using
\eqref{eq:ss4}, we have
\begin{equation}\label{eqn:E}
\begin{aligned}
k\left(E(\Xcal^h)-E(\Ycal)\right)&=\Tr \left(\Xmat^{h\top}\Amat\Xmat^h -\Cmat^{h\top} \Xmat^{h\top} \Amat\Xmat^h\Cmat^h-\Cmat^{l\top} \Xmat^{l\top} \Amat\Xmat^l\Cmat^l \right) \\
&=\Tr \left( \Lambda^h-\Cmat^{h\top} \Lambda^h\Cmat^h-\Cmat^{l\top}\Lambda^l \Cmat^l \right)\,,
\end{aligned}
\end{equation}
\steveedit{where $\Lambda^h := \text{diag} (\lambda_1,\lambda_2,\dotsc,\lambda_k)$ and
$\Lambda^l := \text{diag} (\lambda_{k+1},\lambda_{k+2},\dotsc,\lambda_n)$.}
For the terms on the right-hand side of \eqref{eqn:E}, we have
\begin{equation}
\label{eqn:lambdak+1}
\Tr \left(\Cmat^{l\top}\Lambda^l\Cmat^l \right) \leq \lambda_{k+1}
\Tr \left(\Cmat^{l\top} \Cmat^l \right) \,,
\end{equation}
and that
\begin{align}
\nonumber
\Tr \left( \Lambda^h-\Cmat^{h\top} \Lambda^h\Cmat^h \right) & = \sum_{j=1}^k \lambda_j - \sum_{i=1}^k \sum_{j=1}^k \lambda_j c_{ji}^2 \\
\nonumber
& = \sum_{j=1}^k \lambda_j  \left(1 - \sum_{i=1}^k c_{ji}^2 \right) \\
\nonumber
& \ge \lambda_k \sum_{j=1}^k \left(1 - \sum_{i=1}^k c_{ji}^2 \right) \\
\label{eqn:lambdak}
& = \lambda_k \Tr \left( \Cmat^{l\top} \Cmat^l \right)
\end{align}
where we  used  \eqref{eq:Csum}.
By substituting~\eqref{eqn:lambdak+1} and~\eqref{eqn:lambdak} into~\eqref{eqn:E}, we obtain
\begin{align*}
k\left(E(\Xcal^h)-E(\Ycal)\right)&\geq (\lambda_k - \lambda_{k+1})  \Tr \left(\Cmat^{l\top} \Cmat^l \right)\,,
\end{align*}
which is equivalent to \eqref{eq:ss7}.

When condition \eqref{eq:sw9} holds, we have from \eqref{eq:ss7} that
$\Tr(\Cmat^{l\top} \Cmat^l) <1$. Thus since $\Cmat^{h\top}\Cmat^h =
\mathbbm{1} - \Cmat^{l\top}\Cmat^l$ and setting $\| \Cmat^{l\top}\Cmat^l
\|\leq \Tr(\Cmat^{l\top}\Cmat^l)<1$, we have that
$\Cmat^{h\top}\Cmat^h$ is nonsingular, so that the $k \times k$ matrix
$\Cmat^h$ is nonsingular.
\end{proof}

We finally use energy distance $E(\Xcal^h)-E(\Ycal)$ to estimate the 
\keedit{Kolmogorov distance} $d(\Ycal,\Xcal^h)$, as follows.

\begin{theorem}\label{thm:energy_to_distance}
\keedit{Considering the optimal space $\Xcal^h$ and the proposed
  space $\Ycal$, if}
\begin{equation} \label{eq:sw13}
E(\Xcal^h) - E(\Ycal) \leq \frac{\lambda_k -\lambda_{k+1} }{2k} \,,
\end{equation}
then we have 
\begin{equation}\label{eqn:thm_est}
	d(\Ycal,\Xcal^h) \leq \sqrt{\lambda_{k+1} \frac{\| \Cmat^{l\top}\Cmat^l \|}{1- \| \Cmat^{l\top}\Cmat^l \|} }\,.
\end{equation}
and furthermore
\begin{equation}\label{eqn:thm_est2}
d(\Ycal,\Xcal^h)\leq \sqrt{2\lambda_{k+1}\frac{k\left(E(\Xcal^h)-E(\Ycal)\right)}{\lambda_k-\lambda_{k+1}} }\,.
\end{equation}
\end{theorem}
\begin{proof}
Choosing an arbitrary $z = \Xmat^h \alpha$ with $\|\alpha\|\leq1$, we
look for $\beta$ such that $y=\Ymat \beta$ is closest to $z$ in
$\Amat$-norm. 
The solution, obtained from the minimization problem
\begin{equation} \label{eq:def.f}
\min_\beta f_{\alpha}(\beta) :=\| y-z\|_\Amat^2 = (\Ymat \beta -\Xmat^h\alpha)^\top \Amat (\Ymat \beta -\Xmat^h\alpha)\,
\end{equation}
 is
\begin{equation} \label{eq:def.betas}
	\beta^\ast_{\alpha} = (\Ymat^\top \Amat \Ymat)^{-1} \Ymat^\top \Amat \Xmat^h \alpha\,.
\end{equation}
\steveedit{Note from the definition \eqref{eqn:distance} that
  \begin{equation} \label{eq:fd}
  d(\Ycal,\Xcal^h) = \sup_{\|\alpha\| \le 1} \sqrt{f_{\alpha} (\beta^{\ast}_{\alpha} )}.
  \end{equation}
} 
From \eqref{eq:ss4} and \eqref{eq:ss3}, we have
\begin{equation} \label{eq:YAY}
\Ymat^\top \Amat \Ymat = \Cmat^\top \Lambda \Cmat =\Cmat^{h \top}
\Lambda^h \Cmat^h + \Cmat^{l \top} \Lambda^l \Cmat^l,
\end{equation}
which is invertible, since $\Cmat$ has orthonormal columns and
$\Lambda$ is diagonal and positive definite. Thus $\beta^\ast_{\alpha}$ is well
defined by \eqref{eq:def.betas}. By substituting \eqref{eq:def.betas}
into \eqref{eq:def.f}, we obtain
\begin{equation}\label{eqn:min}
f_{\alpha}(\beta_{\alpha}^\ast) = -\alpha^\top (\Ymat^\top \Amat \Xmat^h)^\top
(\Ymat^\top \Amat \Ymat)^{-1} (\Ymat^\top \Amat \Xmat^h)\alpha
+\alpha^\top \Xmat^{h\top} \Amat \Xmat^ h \alpha.
\end{equation}
Note from \eqref{eq:ss3} and  \eqref{eq:sw65} that
\[
\Amat \Ymat = \Amat \Xmat^h \Cmat^h + \Amat \Xmat^l \Cmat^l =
\Bmat \Xmat^h \Lambda^h \Cmat^h + \Bmat \Xmat^l \Lambda^l \Cmat^l,
\]
so from \eqref{eq:ss4}, we have
\[
\Xmat^{h \top} \Amat \Ymat =  (\Xmat^{h \top} \Bmat \Xmat^h) \Lambda^h \Cmat^h + (\Xmat^{h \top} \Bmat \Xmat^l) \Lambda^l \Cmat^l =
\Lambda^h \Cmat^h.
\]
By substituting this equality together with \eqref{eq:YAY} into
\eqref{eqn:min}, \steveedit{and using \eqref{eq:ss4} again,} we have
\begin{equation}\label{eqn:min.a}
f_{\alpha}(\beta^\ast_{\alpha}) = -\alpha^\top (\Lambda^h \Cmat^h) (\Cmat^{h\top}
\Lambda^h \Cmat^h + \Cmat^{l\top}\Lambda^l \Cmat^l)^{-1} (\Lambda^h
\Cmat^h)^\top \alpha +\alpha^\top \Lambda^h \alpha\,.
\end{equation}
Invertibility of $\Cmat^h$ follows from Lemma~\ref{lemma:trace_CC} and
the condition \eqref{eq:sw13}, so that $(\Lambda^h)^{1/2} \Cmat^h$ is
invertible, and we can transform \eqref{eqn:min.a} to
\begin{equation} \label{eq:min.a2}
f_{\alpha}(\beta^\ast_{\alpha}) = -\alpha^\top  (\Lambda^h)^{1/2}  \left[
\mathbbm{1} + (\Cmat^{h\top} (\Lambda^h)^{1/2})^{-1} \Cmat^{l\top}\Lambda^l \Cmat^l
((\Lambda^h)^{1/2} \Cmat^h)^{-1} \right]^{-1} (\Lambda^h)^{1/2} \alpha +
\alpha^T \Lambda^h \alpha.
\end{equation}
For any matrix $\Amat$ with $(\mathbbm{1}+\Amat)$ nonsingular, we have
that
\begin{equation} \label{eq:sw18}
(\mathbbm{1}+\Amat)^{-1} = \mathbbm{1} - \Amat + (\mathbbm{1}+\Amat)^{-1} \Amat^2.
\end{equation}
Moreover, if $\Amat$ is symmetric positive semidefinite, the last term
$(\mathbbm{1}+\Amat)^{-1} \Amat^2$ is symmetric positive semidefinite,
since if we write the eigenvalue decomposition of $\Amat$ as $\Amat =
\Umat \Smat \Umat^{\top}$ where $\Umat$ is orthogonal and $\Smat$ is
nonnegative diagonal, we have that $(\mathbbm{1}+\Amat)^{-1} \Amat^2 =
\Amat^2 (\mathbbm{1}+\Amat)^{-1} = \Umat (\mathbbm{1}+\Smat)^{-1}
\Smat^2 \Umat^{\top}$. Thus for any vector $z$, we have from
\eqref{eq:sw18} that
\[
-z^\top (\mathbbm{1}+\Amat)^{-1} z + z^Tz \le -z^T(\mathbbm{1}-\Amat) z + z^Tz =
z^T \Amat z.
\]
By substituting $\Amat = (\Cmat^{h\top} (\Lambda^h)^{1/2})^{-1}
\Cmat^{l\top}\Lambda^l \Cmat^l ((\Lambda^h)^{1/2} \Cmat^h)^{-1}$ and
$z = (\Lambda^h)^{1/2} \alpha$ into this expression, we have from
\eqref{eq:min.a2} that
\begin{align} 
\nonumber
f_{\alpha}(\beta^\ast_{\alpha}) & \steveedit{\le} \alpha^\top  (\Lambda^h)^{1/2} (\Cmat^{h\top} (\Lambda^h)^{1/2})^{-1}
\Cmat^{l\top}\Lambda^l \Cmat^l ((\Lambda^h)^{1/2} \Cmat^h)^{-1} (\Lambda^h)^{1/2} \alpha \\
\nonumber
&  = 
\alpha^\top   (\Cmat^h)^{-\top} \Cmat^{l\top}\Lambda^l \Cmat^l (\Cmat^h)^{-1} \alpha  \\
& \le \| \alpha \|^2 \| (\Cmat^{h\top} \Cmat^h)^{-1} \| \| \Cmat^{l\top}\Lambda^l \Cmat^l \|.
\label{eq:min.a3}
\end{align}
Note that $\|\alpha\|\leq 1$, $ \|\Cmat^{l\top}\Lambda^l \Cmat^l\| \leq \lambda_{k+1} \| \Cmat^{l\top}\Cmat^l \|$ and
\begin{align*}
\left\| (\Cmat^{h\top}\Cmat^h)^{-1} \right\| = \left\| (\mathbbm{1}-\Cmat^{l\top}\Cmat^l )^{-1} \right\| \leq \sum_{i=0}^\infty \left\| \Cmat^{l\top}\Cmat^l \right\|^i =\frac{1}{1- \| \Cmat^{l\top}\Cmat \|}\,,
\end{align*}
so by substituting into \eqref{eq:min.a3}, we have
\begin{equation}
f_{\alpha}(\beta^\ast_{\alpha}) \leq \lambda_{k+1} \frac{\| \Cmat^{l\top}\Cmat^l\|}{1-
  \| \Cmat^{l\top}\Cmat^l \|}\,, \quad \mbox{for all $\alpha$ with $\| \alpha\| \le 1$,}
\end{equation}
\steveedit{which because of \eqref{eq:fd} yields \eqref{eqn:thm_est}.}

Under condition \eqref{eq:sw13} we have from
Lemma~\ref{lemma:trace_CC} that
\[
\| \Cmat^{l\top} \Cmat \| \le \Tr (\Cmat^{l\top} \Cmat) \le 
\frac{k\left(E(\Xcal^h)-E(\Ycal)\right)}{\lambda_k-\lambda_{k+1}} \le
\frac12,
\]
so that
\[
\frac{\| \Cmat^{l\top}\Cmat^l \|}{1- \| \Cmat^{l\top}\Cmat^l \|} \le 2 
\| \Cmat^{l\top}\Cmat^l \| \le 2 
 \frac{k\left(E(\Xcal^h)-E(\Ycal)\right)}{\lambda_k-\lambda_{k+1}},
\]
yielding  \eqref{eqn:thm_est2}.
\end{proof}

\subsection{Criteria Used in Random Sampling for Local Basis Functions}\label{sec:basis_energy}

In our local basis construction problem, we identify $\Amat$ and
$\Bmat$ in Section~\ref{sec:highD} with $\Smat$ and $\Smat^\ast$,
respectively, from \eqref{eqn:GEP}. An energy term is constructed
similarly.
\begin{definition}[Energy of a function space]
  \blueedit{Given the function space $\Gamma_n$,
    let
    $\{\tilde{\gamma}_1,\tilde{\gamma}_2,\dotsc,\tilde{\gamma}_n\}$ be
    an $\mathcal{E}(\omega^*)$-orthonormal basis for $\Gamma_n$. The
    energy of $\Gamma_n$ is defined by
\begin{equation}\label{eqn:energy_gamma}
E(\Gamma_n) := \frac{\sum_{i=1}^n\langle a(x)|\nabla_x\tilde{\gamma}_i|^2\rangle_{\CalE(\omega)}}{\sum_{i=1}^n\langle a(x)|\nabla_x\tilde{\gamma}_i|^2\rangle_{\CalE(\omega^*)}}\,.
\end{equation}
}
\end{definition}

According to Theorem~\ref{thm:energy_to_distance}, a larger value of
$E$ indicates a smaller angle to the optimal basis set, and thus a
better sampling strategy.


%

\blueedit{Theorem~\ref{thm:energy_to_distance} suggests that a
    sampling strategy that provides a matrix $\Ymat$ of discretized
    basis functions with higher energy $E(\Ycal)$ (closer to the
    optimal value of $E(\Xcal^h)$) will result in a smaller 
    \blueedit{Kolmogorov distance}, and thus a better approximation to the optimal space
    $\Xcal^h$.
    Larger values of $E$ are achieved when the
    samples have their energies largely supported in the
    interior. This further suggests that construction of $a$-harmonic
    functions through random sampling of singular boundary conditions
    may not be the best strategy, because the boundary layer close to
    $\omega^\ast$ quickly damps out the solution and the energies
    concentrated in the margin $\omega^\ast \char`\\ \omega$, leading
    to relatively small energy in the interior and a smaller value of
    $E$ in \eqref{eqn:energy_gamma}.  These observations are borne out
    by the numerical experiments reported in the next section.
    Sampling strategies that avoid boundary layers are thereby
    preferred, which suggests that Random Gaussian (strategy 4) and
    Random smooth boundary sampling (strategy 7) are likely to give
    better results. Our computational results support this
    claim. 
}

\blueedit{
We note that enlarging over-sampling size is another efficient way of
getting rid of boundary layers, but that generally leads to a smaller
Energy value in~\eqref{eqn:energy_gamma}.
}

\begin{theorem}
In a two-dimensional space, 
\keedit{for any small $\varepsilon>0$ and $n$ sufficiently large,}
given any $u\in H_a(\omega^\ast)/\mathbb{R}$
and a subspace $\Gamma_n=\Span \{\gamma_1,\gamma_2,\dotsc,\gamma_n\}$
spanned by random samples of $a$-harmonic functions, the accuracy of
approximating $u$ with a function $\gamma$ from $\Gamma_n$ is bounded
by the following estimate:
\begin{equation*}
\min_{\gamma\in\Gamma_n} \, \|u-\gamma\|_{\CalE(\omega)} \leq \|u\|_{\CalE(\omega^\ast)} \left( e^{-n^{1/3-\varepsilon}} + d(\Gamma_n,\Psi_n) \right)\,,
\end{equation*}
where $\Psi_n$ is defined in~\eqref{eqn:def_Psi} and $d(\Gamma_n,\Psi_n)$ is defined in equation~\eqref{eqn:distance}.
\end{theorem}
\begin{proof}
Without loss of generality, we assume $\|u\|_{\CalE(\omega^\ast)}=1$.
Consider the optimal basis $\{\psi_i\}_{i=1}^\infty\subset
H_a(\omega^\ast)/\mathbb{R}$ computed in~\eqref{eqn:GEP_P}, 
for which we have
\[
\langle \psi_i,\psi_j \rangle_{\CalE(\omega^\ast)} = \delta_{ij}\,,\quad \langle \psi_i,\psi_j \rangle_{\CalE(\omega)} = \lambda_i \delta_{ij}.
\]
We therefore have scalars $u_1,u_2, \dotsc$ such that
\[
u = \sum_{i=1}^\infty u_i \psi_i,\quad \sum_{i=1}^\infty u_i^2 = 1\,.
\]
Defining $\tilde{u}\in H_a(\omega^\ast)/\mathbb{R}$ by
\[
\tilde{u} = \sum_{i=1}^n u_i\psi_i \,,
\]
the restriction $v = P\tilde{u}\in H_a(\omega)/\mathbb{R}$ has
\[
\|u-v\|_{\CalE(\omega)} = \|u-P\tilde{u}\|_{\CalE(\omega)} =\biggl(\sum_{i=n+1}^\infty u_i^2 \lambda_i\biggr)^{1/2}\,.
\]
By the definition \eqref{eqn:distance} of \keedit{Kolmogorov
  distance}, there exists $\gamma\in \Gamma_n$ such that
\[
\|v-\gamma\|_{\CalE(\omega)} \leq \|\tilde{u}\|_{\CalE(\omega^\ast)} d(\Gamma_n,\Psi_n)\,,
\]
where $\Psi_n:=\Span \{\psi_1,\ldots,\psi_n\}$ as in  \eqref{eqn:def_Psi}.
We further note that
\[
\|\tilde{u}\|_{\CalE(\omega^\ast)} = \biggl(\sum_{i=1}^n u_i^2\biggr)^{1/2}\,,
\]
and therefore
\begin{align*}
\|u-\gamma\|_{\CalE(\omega)}&\leq \|u-v\|_{\CalE(\omega)}+\|v-\gamma\|_{\CalE(\omega)}\\
&\leq \biggl(\sum_{i=n+1}^\infty u_i^2\lambda_i\biggr)^{1/2} +\biggl(\sum_{i=1}^n u_i^2\biggr)^{1/2}\, d(\Gamma_n,\Psi_n)\\
&\leq (\lambda_{n+1})^{1/2} +d(\Gamma_n,\Psi_n)\\
&\leq e^{-n^{1/3-\varepsilon}} +d(\Gamma_n,\Psi_n)\,,
\end{align*}
where the last inequality comes from Theorem \ref{thm:accuracy}. 
\end{proof}

\section{Computational Results}\label{sec:numerics}

\keedit{
	We present numerical results in this section that show how the
	Kolmogorov distance of the random sampling subspace to the optimal space 
	decreases with the number of basis
	functions, for different sampling strategies. }
Throughout the section,
the domain $\omega$ and enlarged domain $\omega^\ast$ are defined by
\[
\omega = [-1,1]\times [-1,1], \quad
\omega^\ast = [-1.4,1.4]\times [-1.4,1.4].
\] 
The media $a(x,y)$ is defined to be
\begin{multline}
a(x,y) = \dfrac{1}{5}\left(\dfrac{1.1+\sin(7\pi x)}{1.1+\sin(7\pi y)}+\dfrac{1.1+\sin(9\pi y)}{1.1+\cos(9\pi x)}+\dfrac{1.1+\cos(13\pi y)}{1.1+\cos(13\pi x)}\right.\\
\left.+\dfrac{1.1+\cos(9\pi x)}{1.1+\sin(9\pi y)}+\dfrac{1.1+\sin(7\pi y)}{1.1+\sin(7\pi x)}\right),\quad  (x,y)\in \omega^\ast\,.
\end{multline}
Numerical results will be shown for discretization parameters
$\rd{x}=\rd{y}=1/40$.

The reference solution is obtained from the procedure summarized in
Section~\ref{sec:local_basis}. To find the optimal solution space, we
prepare the entire $a$-harmonic function space by going through all
possible boundary condition configurations, before computing the
general eigenvalue problem~\eqref{eqn:GEP} for basis selection. This
process requires computation of the elliptic
equation~\eqref{eqn:delta} $444$ times (each time with a hat function
on the boundary of $\partial\omega^\ast$ concentrated at a specific
grid point) followed by computation of the generalized eigenpairs of
two matrices of size $444\times 444$.
\keedit{
We then implement all random sampling methods proposed in Section~\ref{sec:methods}. As we see below the seven strategies have varying degrees of efficiency, but all capture the low rank structure of the optimal space.
}

\subsection*{High-Energy Modes}

The first four modes $\{\phi_{1,2,3,4}\}$ of the reference solution
are shown in Figure~\ref{fig:FirstFourModes}. These are obtained by
following the procedure described in Section~\ref{sec:local_basis}. We
note here the presence of boundary layers in $\omega^\ast$, as the
functions exhibit fine scale oscillations near the boundary $\partial
\omega^{\ast}$; moreover, these oscillations in the boundary layer are
\steveedit{trimmed away when the functions are confined to the patch
  $\omega$.}

\begin{figure}[htbp]
  \centering
  \includegraphics[width=\textwidth,height = 200pt]{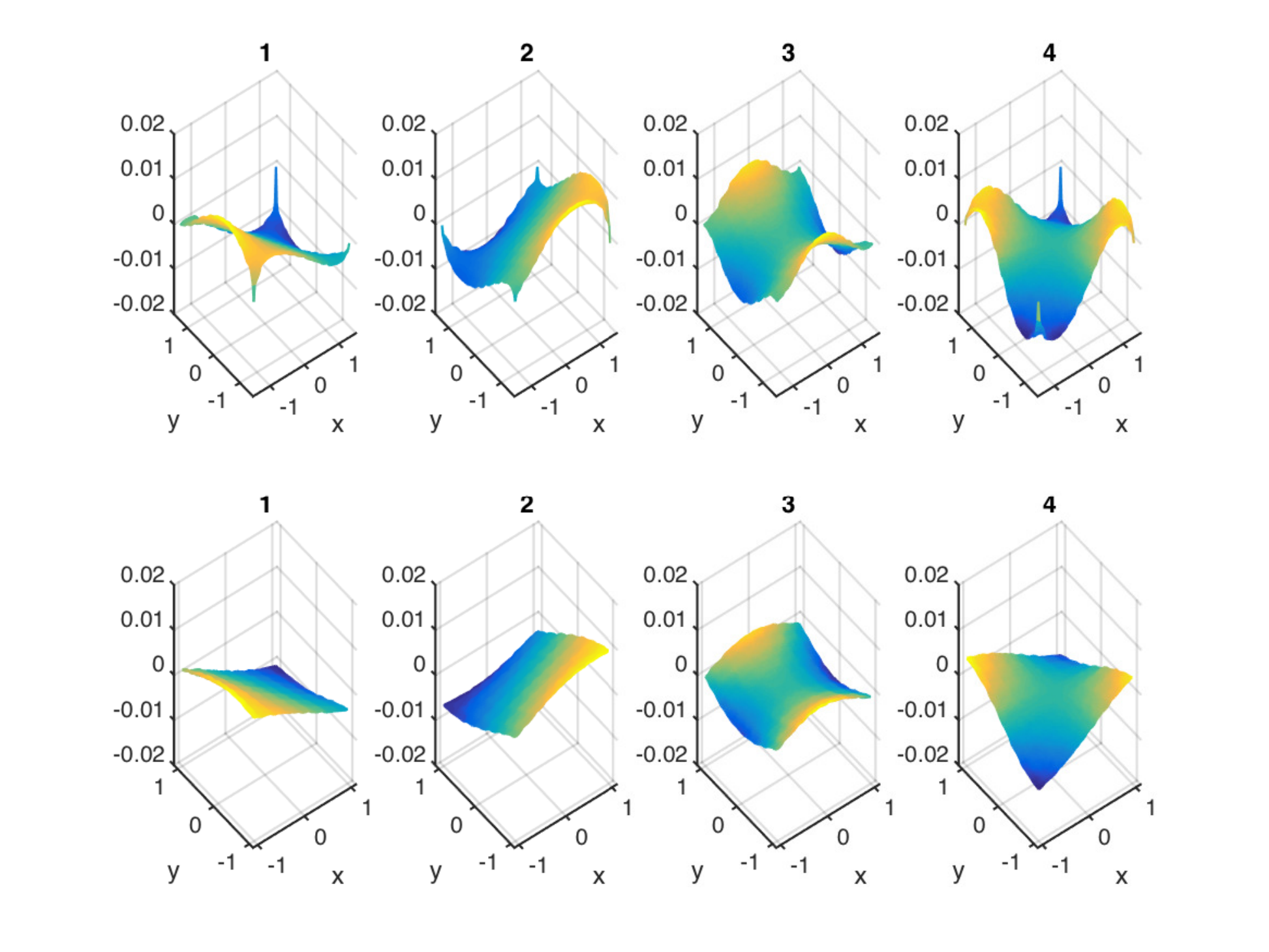}
  \caption{The first row shows the four modes $\phi_{1,2,3,4}$ supported on $\omega^\ast$, and the second row shows the same modes confined in $\omega$. \steveedit{Note that the boundary layers that appear in $\omega^{\ast}$ are not evident in $\omega$.}}
  \label{fig:FirstFourModes}
\end{figure}


\subsection*{Recovery of General Eigenvalues}

We now describe results obtained by random sampling method with the
\steveedit{seven} sampling strategies. For each strategy, we sample only
$20$ $a$-harmonic functions for the computation in
equation~\eqref{eqn:ran_GEP}, hoping that these $20$ random samples
still capture the highest energy modes. In
Figure~\ref{fig:EnergyDecay20}, we plot (in log scale) the $20$
generalized eigenvalues obtained from each of the \keedit{seven}
sampling strategies, together with the leading 20 eigenvalues from the
optimal reference solution.  \blueedit{All methods give almost
    exponential decay of the eigenvalues. By far, the Random Gaussian
    and Smooth Boundary (Strategy 5 and 7) are the best two strategies in tracking the
    eigenvalues of the reference solution.  }

It is expected that as the number of random samples increases, all
random methods should do better at capturing the eigenvalues of the
reference solution. \steveedit{This phenomenon is evident} in
Figure~\ref{fig:EnergyDecay300}, where we use $300$ random samples for
all \keedit{seven} sampling strategies. All except the strategies
involving the full-domain i.i.d.~function and possibly the boundary
i.i.d.~function do well at matching the reference eigenvalues.


\begin{figure}[htbp]
  \centering
  \includegraphics[width=0.45\textwidth]{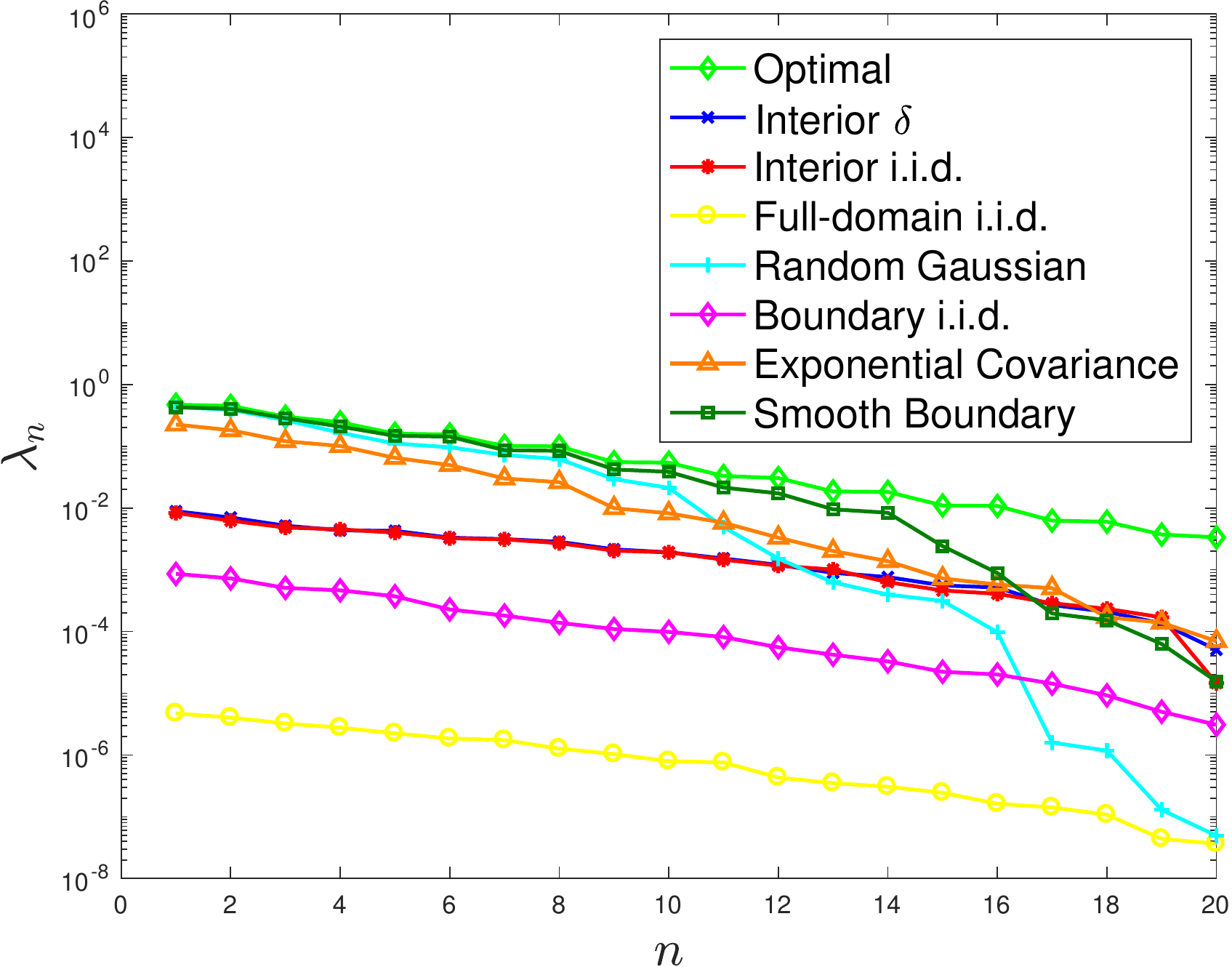}
  \includegraphics[width=0.45\textwidth]{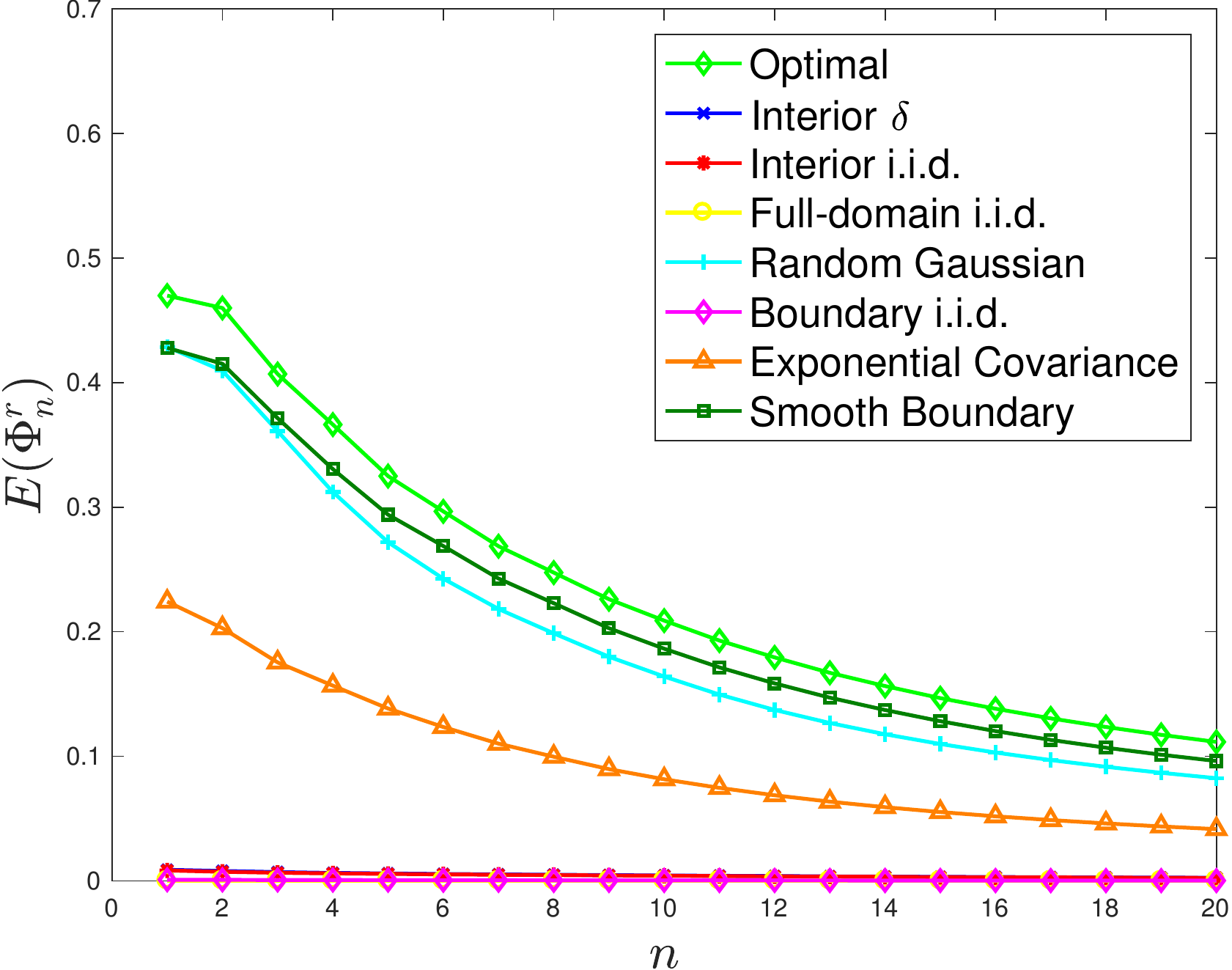}
  \caption{ \blueedit{\textbf{Left:} Eigenvalues obtained from
        the six different random sampling strategies, using $20$
        samples each, and the leading $20$ eigenvalues from the
        reference solution. Eigenvalues are computed from
        \eqref{eqn:ran_GEP} for the random sampling strategies and
        \eqref{eqn:GEP} for the reference solution. All methods show
        almost exponential decay of the eigenvalues, and the Random
        Gaussian and the Smooth Boundary are the best two sampling
        strategies in the sense that they match the reference
        eigenvalues most closely. \textbf{Right:} Energy $E(\Phi_n^r)$
        of optimal space and approximate subspaces from different
        random sampling strategies using $20$ samples. Energy is
        computed from \eqref{eqn:energy_gamma}. Again, Random Gaussian
        and the Smooth Boundary strategy achieve the minimal energy
        gap from the optimal space. } }
  \label{fig:EnergyDecay20}
\end{figure}

\begin{figure}[htbp]
  \centering
  \includegraphics[width=0.45\textwidth]{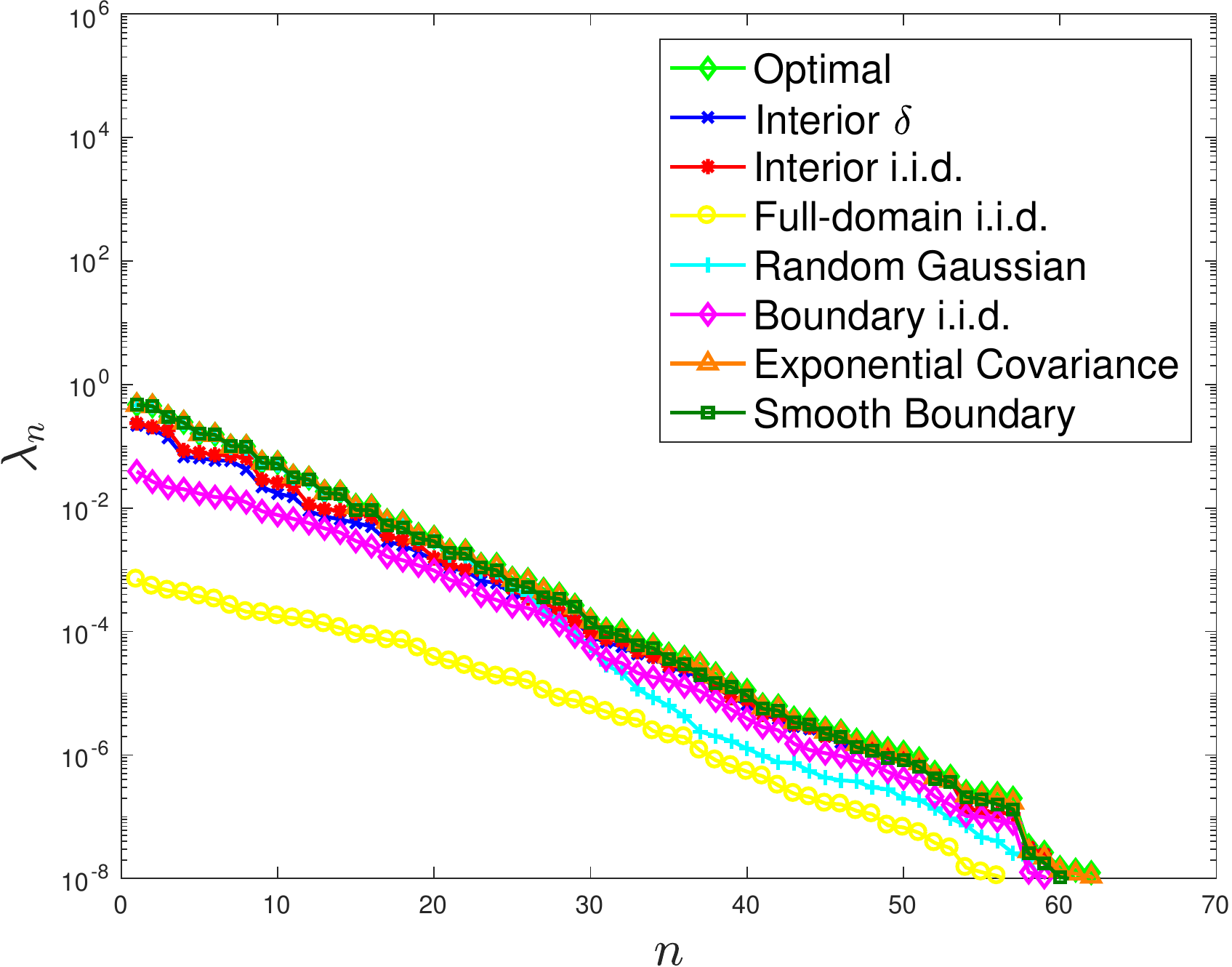}
  \includegraphics[width=0.45\textwidth]{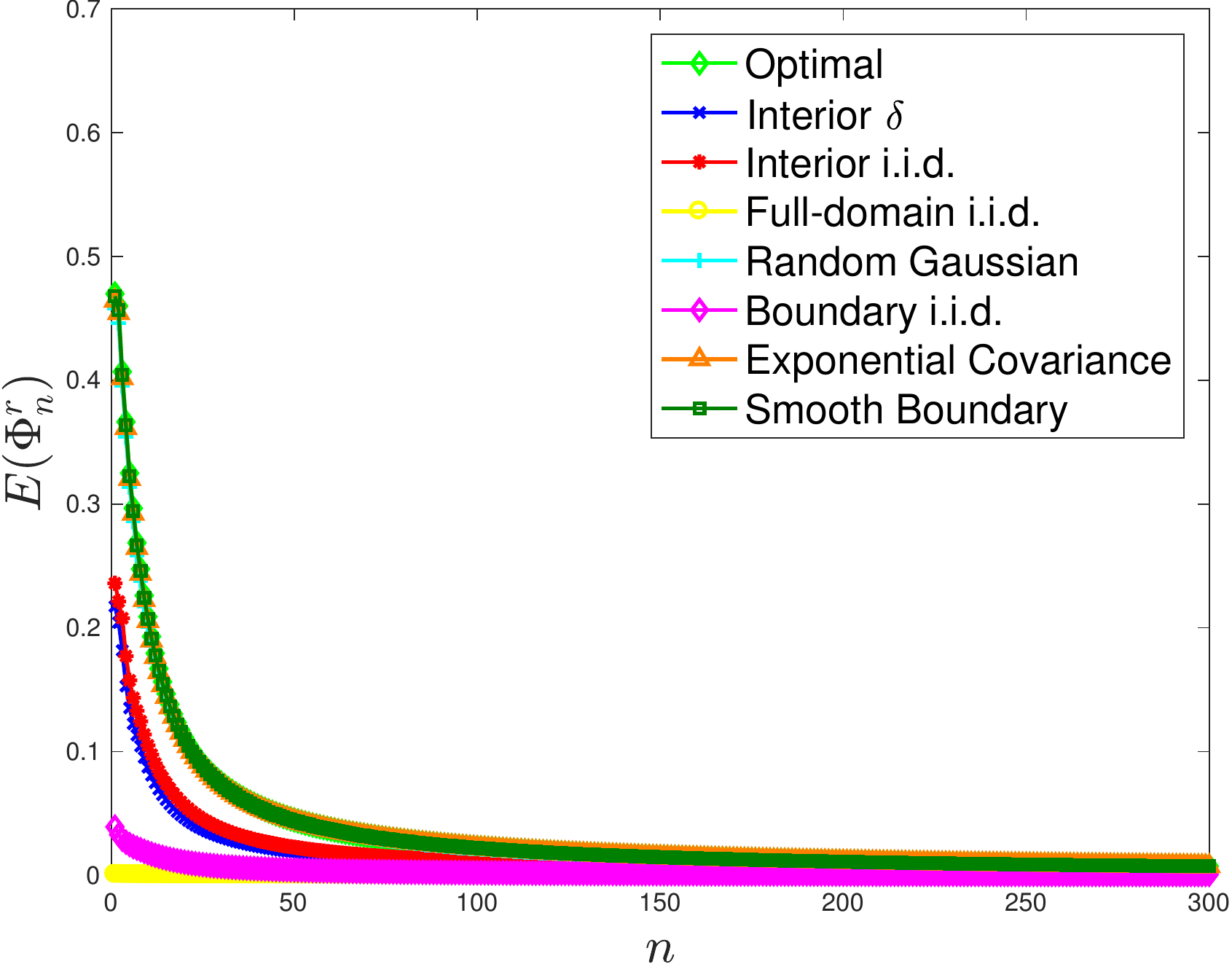}
  \caption{ \blueedit{Same as Figure~\ref{fig:EnergyDecay20}, but
    with $300$ random samples instead of $20$. Since $N_r\sim N_y$,
    the eigenvalues and energies obtained from random sampling tend to match the
    reference eigenvalues and optimal energy more closely. } }
  \label{fig:EnergyDecay300}
\end{figure}

Figure~\ref{fig:DistanceCompare20} shows the recovery of eigenspace by
random sampling procedures. We regard $\Phi_5$, defined
in~\eqref{eq:def.Phin},
as the optimal space (the space expanded by the five modes with
highest energies), and use $\Phi^r_m$ defined
in~\eqref{eqn:RanOptimalSpace} to approximate it, for
$m=5,6,\dotsc,20$. The vertical axis shows 
Kolmogorov distance
whose computation is described in Appendix~\ref{sec:appendixA}. As expected,
using more random modes leads to better recovery and thus smaller
Kolmogorov distance~\eqref{eqn:distance}. The plots show 
Kolmogorov distance decays roughly exponentially fast with respect to $m$, for all five
sampling strategies.  \blueedit{ Once again, the Random Gaussian
    and Smooth Boundary strategy are by far the best: $\Phi^r_{20}$
    approximates $\Phi_5$ with accuracy near $10^{-4}$ or
    $10^{-5}$. The other four strategies attain accuracies of around
    $10^{-1}$ to $10^{-2}$ for $m=20$.  }


\begin{figure}[htbp]
  \centering
  \includegraphics[width=0.6\textwidth]{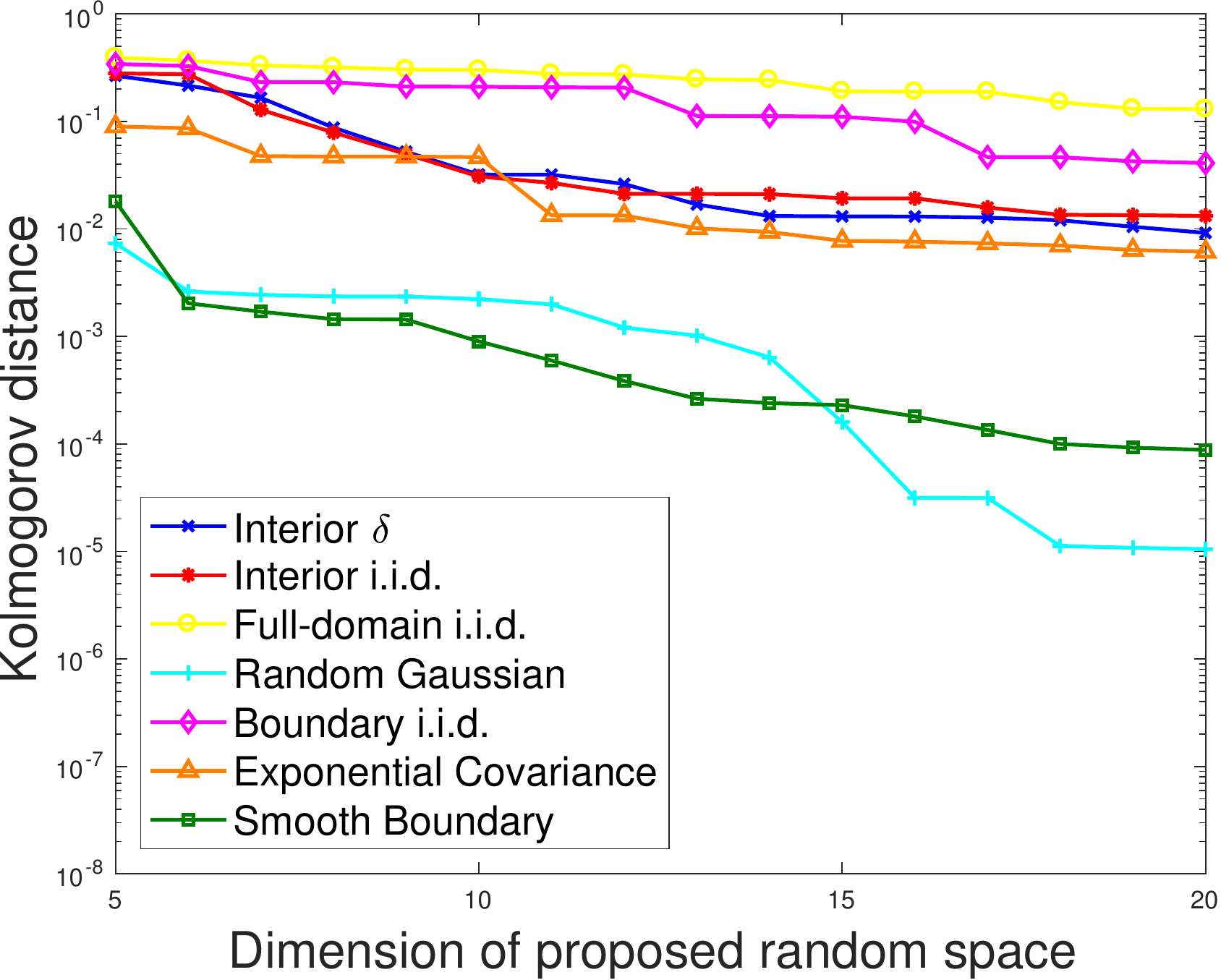}
  \caption{\keedit{Kolmogorov distance of $\Phi_5$ to $\Phi^r_m$, for
		$m=5,6,\dotsc,20$ sampling modes.}}
  \label{fig:DistanceCompare20}
\end{figure}

\subsection*{Eigenspace Recovery for the Random Gaussian Strategy}

Finally, we focus on the Random Gaussian sampling strategy, which is
clearly one of the most successful strategies. In
Figure~\ref{fig:RanGauEigFun}, we plot in the first row the high
energy modes $\{\phi_{1,2,3,4}\}$ for the reference solution, and in
the second row we plot the high energy modes $\{\phi^r_{1,2,3,4}\}$
obtained from the Random Gaussian strategy with $20$ samples. The
similarity is evident.

\begin{figure}[htbp]
  \centering
  \includegraphics[width=\textwidth,height = 200pt]{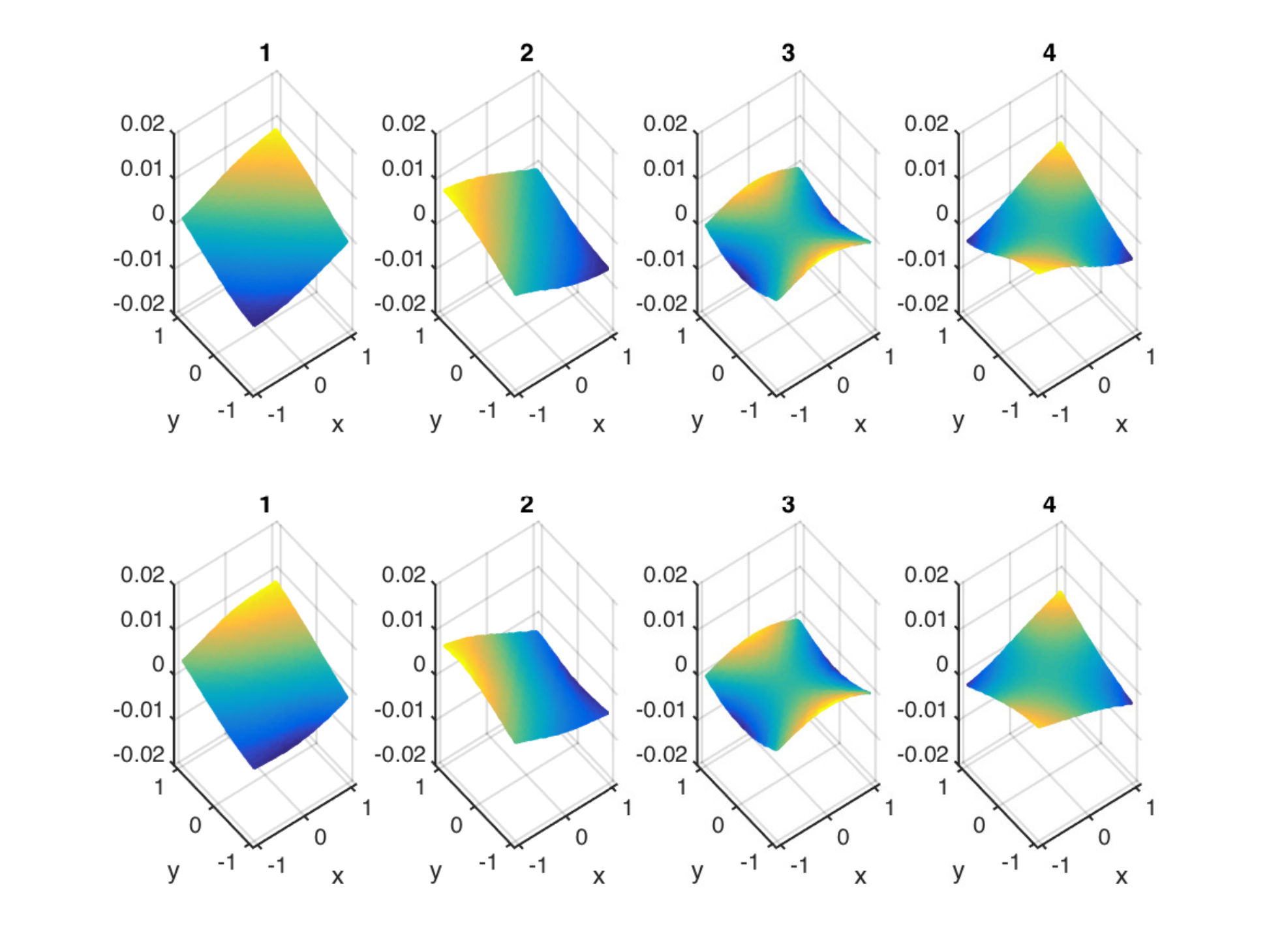}
  \caption{Recovery of high energy modes using Random Gaussian
    sampling strategy, with 20 samples. First row shows the first four
    modes of the reference solution; second row shows the first four
    modes obtained from the Random Gaussian sampling strategy.}
  \label{fig:RanGauEigFun}
\end{figure}


\section{Conclusion}

In this paper we study random sampling methods that approximate the
optimal solution space that attains Kolmogorov $n$-width in the
context of generalized finite element methods.  It is shown that
certain random sampling methods capture the main part of the local
solution spaces with high accuracy, and that efficiency can be
evaluated by the energy contained in the proposed random space.

\blueedit{ Numerical comparisons of seven different sampling strategies
  show that two strategies are superior: Random Gaussian sampling and
  Smooth boundary sampling.  } 

\blueedit{
\section*{Acknowledgments}
We are grateful to the two referees for their constructive comments that improved the paper considerably.}

\begin{appendix}
\section{Calculation of the Kolmogorov Distance}
\label{sec:appendixA}

Suppose we are given the optimal space 
\[
\Phi_m=\Span \Xmat, \; \mbox{where} \; \Xmat = \left[ \phi_1,\ldots,\phi_m \right], 
\]
and a proposed space
\[
\Phi^r_n=\Span \Ymat, \; \mbox{where} \; \Ymat = \left[ \phi^r_1,\ldots,\phi^r_n \right], \quad \mbox{where $n\geq m$},
\]
such that 
\[
\langle \phi_i,\phi_j\rangle_{\mathcal{E}(\omega^\ast)}=\delta_{ij} \quad \text{and}\quad \langle \phi_i^r,\phi_j^r\rangle_{\mathcal{E}(\omega^\ast)}=\delta_{ij}\,.
\] 
Recall the following definition of 
Kolmogorov distance from
\eqref{eqn:distance}:
\[
\keedit{d(\Phi^r_n,\Phi_m)} = \max_{\substack{x\in \Phi_m,\\ \|x\|_{\mathcal{E}(\omega^\ast)\leq 1}}} \,
\min_{y\in \Phi^r_n} \|x-y\|_{\mathcal{E}(\omega)}\,.
\]
To calculate $d(\Phi^r_n,\Phi_m)$ explicitly, we write $x=\Xmat\alpha$ for
some $\alpha\in\mathbb{R}^m$ and $y =\Ymat\beta$ for some
$\beta\in\mathbb{R}^n$.  The 
Kolmogorov distance is achieved when
$\|x\|_{\mathcal{E}(\omega^\ast)}=1$, which implies that
$\|\alpha\|=1$, \steveedit{where this $\|\cdot\|$ is the usual
  Euclidean norm on $\mathbb{R}^m$.} By expanding the objective, we have
\[
\begin{aligned}
\frac{1}{2}\|x-y\|_{\mathcal{E}(\omega)}^2 &=\frac{1}{2}\langle \Xmat\alpha-\Ymat\beta, \Xmat\alpha-\Ymat\beta\rangle_{\mathcal{E}(\omega)}\\
& = \frac{1}{2}\beta^\top \Ymat_A\beta -\alpha^\top C_A \beta +\frac{1}{2}\alpha^\top \Xmat_A \alpha\,,
\end{aligned}
\]
where $(\Ymat_{A})_{ij}=\langle
\phi^r_i,\phi^r_j\rangle_{\mathcal{E}(\omega)}$, $(\Xmat_{A})_{ij}=\langle
\phi_i,\phi_j\rangle_{\mathcal{E}(\omega)}$ and $(C_{A})_{ij}=\langle
\phi_j,\phi^r_i\rangle_{\mathcal{E}(\omega)}$. The minimizing value of $\beta$ is 
given explicitly by
\[
\beta = \Ymat_A^{-1}C_A \alpha\,,
\]
for which we have
\[
\frac{1}{2}d(\Phi^r_n,\Phi_m)^2 = \max_{\|\alpha\|=1} \alpha^\top (\Xmat_A-C_A^\top \Ymat_A^{-1}C_A)\alpha = \|\Xmat_A-C_A^\top \Ymat_A^{-1}C_A\|_{2}^2\,.
\]
The
Kolmogorov distance is therefore
\[
d(\Phi^r_n,\Phi_m) = \sqrt{2} \|\Xmat_A-C_A^\top \Ymat_A^{-1}C_A\|_{2}\,.
\]

{\color{blue}{
\section{Well-posedness of energy $E(\Zcal)$}
\label{sec:appendixB}
We show that energy defined in Definition~\ref{def:energy} is a
well-defined quantity. More specifically, given a $k$-dimensional space $\Zcal$ and two different $\Bmat$-orthonormal matrices $\Zmat_1,\Zmat_2\in\mathbb{R}^{n\times k}$ whose columns span the space $\Zcal$, we show that they yield the same value of $E(\Zcal)$:
	\begin{equation}\label{eqn:well-def}
E(\Zcal)=\frac{\Tr(\Zmat_1^\top \Amat\Zmat_1)}{\Tr (\Zmat_1^\top \Bmat\Zmat_1)}=\frac{\Tr(\Zmat_2^\top \Amat\Zmat_2)}{\Tr (\Zmat_2^\top \Bmat\Zmat_2)}\,.
\end{equation}
\begin{proof}
Since $\Zmat_1$ and $\Zmat_2$ share the column space $\Zcal$,
there must exist an invertible matrix $\Pmat\in \mathbb{R}^{k\times
  k}$ such that $\Zmat_1 = \Zmat_2 \Pmat$. 

We show first that $\Pmat$ is unitary. Because both $\Zmat_1$ and
$\Zmat$ have $\Bmat$-orthonormal columns, we have
\begin{equation*}
\Zmat^\top_1\Bmat \Zmat_1 = \Zmat^\top_2\Bmat \Zmat_2 = \Imat \,,
\end{equation*}
which implies that
\begin{equation*}
\Imat = \Zmat^\top_1\Bmat \Zmat_1 = (\Zmat_2\Pmat)^\top \Bmat (\Zmat_2\Pmat) =
  \Pmat^\top\Zmat^\top_2\Bmat \Zmat_2 \Pmat =\Pmat^\top
  \Pmat.
\end{equation*}
By definition of $\Zmat_1$ and $\Zmat_2$, we have
\[
\frac{\Tr(\Zmat^\top_1\Amat\Zmat_1)}{\Tr(\Zmat^\top_1\Bmat\Zmat_1)} =  \frac{\Tr(\Zmat^\top_1\Amat\Zmat_1)}{k}, \quad
\frac{\Tr(\Zmat_2^\top \Amat\Zmat_2)}{\Tr (\Zmat_2^\top \Bmat\Zmat_2)} = \frac{\Tr(\Zmat^\top_2\Amat\Zmat_2)}{k},
\]
so our claim \eqref{eqn:well-def} will hold if we can show that
$\Tr(\Zmat^\top_1\Amat\Zmat_1) = \Tr(\Zmat^\top_2\Amat\Zmat_2)$.
But this follows from
\begin{equation*}
\Tr (\Zmat^\top_1\Amat\Zmat_1) = \Tr\left((\Zmat_2\Pmat)^\top \Amat
(\Zmat_2\Pmat)\right) =
\Tr\left(\Pmat^\top\Zmat_2^\top\Amat\Zmat_2\Pmat\right) =
\Tr\left(\Zmat_2^\top\Amat\Zmat_2\Pmat\Pmat^\top\right) =
\Tr(\Zmat_2^\top \Amat\Zmat_2),
\end{equation*}
where the last equality comes from the orthonormality of $\Pmat$.
\end{proof}
}}
\end{appendix}

\bibliographystyle{siam}
\bibliography{elliptic_random_lowrank}
\end{document}